\documentclass[GL4 moment]{article}
\usepackage{amsmath, amssymb, amsthm, delarray}

\usepackage[top =20mm, left=35mm,right=37mm]{geometry}

\newtheorem{thm}{Theorem}[section]
\newtheorem*{theorem*}{Theorem}
\newtheorem{lem}[thm]{Lemma}
\newtheorem{prop}[thm]{Proposition}

\theoremstyle{definition}

\newtheorem{rem}[thm]{Remark}

\numberwithin{equation}{section}
\theoremstyle{remark}

\begin{document}

\title{\textbf{The second moment of Dirichlet twists of a $\textrm{GL}_{4}$ automorphic $L$-function}}

\author{Keiju Sono}

\date{}
\allowdisplaybreaks

\maketitle 
\noindent
\begin{abstract}
In this paper, we give an asymptotic formula for the second moment of Dirichlet twists of an automorphic $L$-function $L(s, \pi)$ on the critical line averaged over characters and conductors, where  $\pi$ denotes an irreducible tempered cuspidal automorphic representation of $\mathrm{GL}_{4}(\mathbb{A}_{\mathbb{Q}})$ with unitary central character.  We give some hybrid bound for the error term with respect to the size of conductors of Dirichlet characters and that of the automorphic representation.

\end{abstract}


\section{Introduction}
Estimation of the moments of the Riemann zeta-function or other $L$-functions has been regarded as a very important problem in analytic number theory. It is related to some fundamental problems in number theory, for example, estimation of the ranks of elliptic curves, zero density estimates for $L$-functions, the Lindel{\"o}f hypothesis and so on. For $k>0$, the $2k$ th moment of the Riemann zeta-function $\zeta (s)$ on the critical line $\Re (s)=1/2$ is defined by 
\[
I_{k}(T):=\int _{1}^{T}\left| \zeta (1/2+it ) \right|^{2k} dt
\]
for $T>1$. In 1918, Hardy and Littlewood \cite{HL} proved $I_{1}(T)\sim T\log T$, and in 1926 Ingham \cite{I} obtained $I_{2}(T)\sim (1/2\pi ^{2})T\log ^{4}T$. It is generally conjectured that $I_{k}(T)\sim C_{k}T\log ^{k^{2}}T$ for any $k>0$, where $C_{k}$ is some positive constant dependent only on $k$. The value of $C_{k}$ has been conjectured  by Keating and Snaith \cite{KS} using the random matrix theory. Though there are a number of works to compute $I_{k}(T)$, the main term has been obtained only in the cases $k=1,2$. It is widely believed that obtaining an asymptotic formula for $I_{k}(T)$ with $k>2$ is out of the reach of our current techniques. 

Also, there are a lot of studies on moments of other $L$-functions. The main problem is to compute the moments of a class of $L$-functions at the central point $s=1/2$ of their functional equations. For example, Paley \cite{P} obtained an asymptotic formula for the primitive Dirichlet $L$-functions 
\[
\sum _{\chi (\mathrm{mod}\; q)}^{\quad \quad *}\left| L\left(1/2, \chi \right) \right|^{2} \sim \frac{\varphi ^{*}(q) \varphi (q)}{q}\log q
\]
as $q\to \infty$ with $q \not\equiv 2 \; (\textrm{mod}\; 4)$, where $\varphi (q)$ and $\varphi ^{*}(q)$ denote the numbers of Dirichlet characters and primitive Dirichlet characters modulo $q$, respectively.  The asterisk in the summation means that the sum is over primitive characters. The asymptotic formula for the fourth moment of Dirichlet $L$-function was obtained by Heath-Brown \cite{HB} when $q$ does not have many prime factors,  and Soundararajan \cite{Sound} improved Heath-Brown's result. Young \cite{Y} obtained a power saving asymptotic formula for the fourth moment. Blomer, Fouvry, Kowalski, Michel and Mili\'cevi\'c \cite{BFKMM} studied the moment of the product of  Dirichlet twists of modular $L$-functions, and as a corollary they gave a significant improvement on the size of the error term of Young's formula.  Like the case of the Riemann zeta-function, the asymptotic formula of $2k$ th moment of Dirichlet $L$-functions at $s=1/2$ has not been obtained for any $k\neq 1,2$. 

However, by considering both average over conductors and integration on the critical line  in addition to the average over characters of the same moduli, one can obtain some asymptotic formulas for the moments of higher powers of Dirichlet $L$-functions. For example, Conrey, Iwaniec and Soundararajan \cite{CIS} obtained 
\begin{align*}
& \sum _{q\leq Q}\sum _{\chi (\mathrm{mod}\; q)}^{\quad \quad \flat}\int _{-\infty}^{\infty}\left| \Lambda (1/2+iy, \chi ) \right|^{6} dy \\
& \quad \sim 42 a_{3}\sum _{q\leq Q}\prod _{p|q}\frac{(1-1/p)^{5}}{1+4/p+1/p^{2}}\varphi ^{\flat}(q)\frac{(\log q)^{9}}{9!}\int _{-\infty}^{\infty}\left| \Gamma \left( \frac{1/2+iy}{2} \right) \right|^{6} dy
\end{align*}
unconditionally, where 
\[
a_{3}=\prod _{p}(1-1/p)^{4}(1+4/p+1/p^{2})
\]
and $\Lambda (s, \chi)$ is the completed $L$-function defined by 
\begin{equation}
\label{lambda}
\Lambda (1/2+s, \chi)=(q/\pi)^{s/2}\Gamma (1/4+s/2)L(1/2+s, \chi)
\end{equation}
satisfying the functional equation 
\begin{equation}
\label{fe}
\Lambda (1/2+s, \chi)=\varepsilon _{\chi}\Lambda (1/2-s, \overline{\chi}), \quad |\varepsilon _{\chi}|=1.
\end{equation}
The sum above is over even primitive characters and $\varphi ^{\flat}(q)$ denotes the number of even primitive characters modulo $q$. The leading coefficient $a_{3}$ above is found in the conjecture of Keating and Snaith \cite{KS} for the sixth moment of $\zeta (1/2+it)$, i.e., 
\[
\int _{1}^{T}\left| \zeta (1/2+it ) \right|^{6} dt \sim 42 a_{3}T\frac{(\log T)^{9}}{9!}, \quad T \to \infty .
\]
In the case of eighth moment, it is conjectured in \cite{CFKRS} that 
\[
\sum _{\chi (\mathrm{mod}\; q)}^{\quad \quad \flat}\left| L(1/2, \chi ) \right|^{8}\sim 24024 \varphi ^{\flat}(q)a_{4}\prod _{p|q}\frac{(1-1/p)^{7}}{1+9/p+9/p^{2}+1/p^{3}}\frac{(\log q)^{16}}{16!}
\]
as $q \to \infty$ with $q \not\equiv 2 \; (\textrm{mod}\; 4)$, where
\[
a_{4}=\prod _{p}(1-1/p)^{9}(1+9/p+9/p^{2}+1/p^{3}).
\]
Towards this conjecture, Chandee and Li \cite{CL} proved 
\begin{equation}
\label{CL}
\begin{aligned}
& \sum _{q\leq Q}\sum _{\chi (\mathrm{mod}\; q)}^{\quad \quad \flat}\int _{-\infty}^{\infty}\left| \Lambda (1/2+iy, \chi ) \right|^{8} dy \\
& \quad \sim 24024 a_{4}\sum _{q\leq Q}\prod _{p|q}\frac{(1-1/p)^{7}}{1+9/p+9/p^{2}+1/p^{3}}\varphi ^{\flat}(q)\frac{(\log q)^{16}}{16!}\int _{-\infty}^{\infty}\left| \Gamma \left( \frac{1/2+iy}{2} \right) \right|^{8} dy,
\end{aligned}
\end{equation}
under the assumption of the generalized Riemann hypothesis (GRH). They also studied the fourth moment of Dirichlet twists of a $\textrm{GL}_{2}$-automorphic $L$-function. For  a holomorphic modular form  $f$ of weight $k$ and full level, they obtained the asymptotic formula
\begin{equation}
\label{CLmodular}
\begin{aligned}
&\sum _{q}\Psi (q/Q) \sum _{\chi (\mathrm{mod}\; q)}^{\quad \quad *}\int _{-\infty}^{\infty}\left| \Lambda (1/2+iy, f \times \chi ) \right|^{4} dy \\
& \quad =\frac{1}{2\pi ^{2}}f_{2}\sum _{q}\Psi (q/Q)\varphi ^{*}(q)(2\log q)^{4}\prod _{p|q}\frac{1}{B_{p}(f, 1/2)}\int _{-\infty}^{\infty}\left| \Gamma \left( \frac{k/2+iy}{2} \right)\right|^{4} dy \\
& \quad \quad +O(Q^{2}(\log Q)^{3+\varepsilon})
\end{aligned}
\end{equation}
under the assumption of the GRH for each $L(s, f\times \chi)$, where $\Psi$ is a smooth function supported in $[1,2]$, $\Lambda (s, f\times \chi)$ is the completed $L$-function of $L(s, f\times \chi)$,  and $f_{2}$,  $B_{p}(1/2,f)$ are some constants dependent on $f$ given explicitly in \cite{CL}.  As researches in a similar direction, they also studied the sixth moment of automorphic $L$-functions \cite{CL2}, the second moment of $\mathrm{GL}(4) \times \mathrm{GL}(2)$ $L$-functions at special points \cite{CL3},  and the eighth moment of the family of automorphic $L$-functions of $\Gamma _{1}(q)$ \cite{CL4}.

The aim of this paper is to establish an asymptotic formula for the Dirichlet twists of a $\textrm{GL}_{4}$  automorphic $L$-function in a similar situation.  Let $\pi$ be an irreducible cuspidal automorphic representation of $\mathrm{GL}_{4}(\mathbb{A}_{\mathbb{Q}})$ with unitary central character. For $\Re (s)>1$,  let
\begin{equation}
\label{1}
L(s, \pi)=\sum _{n=1}^{\infty}\frac{a_{\pi}(n)}{n^{s}}=\prod _{p}\prod _{j=1}^{4}\left( 1-\frac{\alpha _{j}(p)}{p^{s}} \right)^{-1}
\end{equation}
be the  automorphic $L$-function associated to $\pi$, as defined by Godement and Jacquet in \cite{GJ}.  $L(s,\pi)$ is continued holomorphically to the whole complex plane, and satisfies the functional equation 
\begin{equation}
\label{2}
\Phi (s, \pi):=N_{\pi}^{s/2}\gamma (s,\pi)L(s,\pi)=\varepsilon_{\pi}\overline{\Phi}(1-s,\pi),
\end{equation}
where $|\varepsilon _{\pi}|=1$, $\overline{\Phi}(s,\pi):=\overline{\Phi (\overline{s}, \pi )}$ and $N_{\pi}$ is a positive integer called the conductor of $\pi$. The function $\gamma (s,\pi)$ is called the gamma factor,  given by
\[
\gamma (s,\pi)=\prod _{j=1}^{4}\Gamma _{\mathbb{R}}(s + \mu _{j})
\]
for some $\mu _{j}\in \mathbb{C}$, where $\Gamma _{\mathbb{R}}(s)=\pi ^{-s/2}\Gamma (s/2)$. The temperedness of the  automorphic representation $\pi$ means that the parameters of the Euler product and the functional equation satisfy the following conditions (see \cite{FPRS}): \\
1. {\bf Selberg bound}: $\Re (\mu _{j})\in \{0, 1\}$ for $j=1, \ldots ,4$. \\
2. {\bf Ramanujan bound}: If $p\nmid N_{\pi}$, then  $|\alpha _{j}(p)|=1$ for  $j=1, \ldots ,4$. If $p|N_{\pi}$, then  for each $j=1,\ldots ,4$ either  $\alpha _{j}(p)=0$ or $|\alpha _{j}(p)|=p^{-m_{j}/2}$  ($m_{j}\in \{0,1,2,\ldots \}$) holds.

Let $\chi$ be a primitive Dirichlet character modulo $q$ with $(q, N_{\pi})=1$. Then the twisted $L$-function $L(s, \pi \otimes \chi)$ is defined by
\begin{equation}
\label{3}
L(s, \pi \otimes \chi)=\sum _{n=1}^{\infty}\frac{a_{\pi}(n)\chi (n)}{n^{s}}=\prod _{p}\prod _{j=1}^{4}\left(1-\frac{\alpha _{j}(p)\chi (p)}{p^{s}} \right)^{-1}
\end{equation}
for $\Re (s)>1$. The twisted $L$-function $L(s,\pi \otimes \chi)$ is also continued holomorphically to the whole complex plane as an entire function of order $1$, and satisfies the functional equation 
\begin{equation}
\label{4}
\Phi (s, \pi \otimes \chi):=(q^{4}N_{\pi})^{s/2}\gamma _{\chi}(s, \pi)L(s, \pi \otimes \chi)=\varepsilon _{\pi ,\chi}\overline{\Phi}(1-s, \pi \otimes \chi ),
\end{equation}
where $|\varepsilon _{\pi ,\chi}|=1$ and the function $\gamma _{\chi}(s, \pi)$ is given by 
\[
\gamma _{\chi}(s, \pi)=\prod _{j=1}^{4}\Gamma _{\mathbb{R}}(s+\mu _{j,\chi})
\]
for some $\mu _{j,\chi}\in \mathbb{C}$. We put 
\begin{equation}
\label{5}
\Lambda (1/2+s, \pi \otimes \chi)=(q^{2}N_{\pi}^{1/2}/\pi ^{2})^{-1/2}\Phi (s+1/2, \pi \otimes \chi).
\end{equation}
Then the functional equation (\ref{4}) yields
\begin{equation}
\label{6}
\Lambda (1/2+s, \pi \otimes \chi)=\varepsilon_{\pi ,\chi}\Lambda (1/2-s, \overline{\pi \otimes \chi}),
\end{equation}
where $\overline{\pi \otimes \chi}:=\tilde{\pi}\otimes \overline{\chi}$. Here, $\tilde{\pi}$ denotes the contragredient representation of $\pi$. Let $\Psi$ be a smooth function compactly supported in $[1,2]$. Put $\mathfrak{m}:=\sum _{j=1}^{4}\Re (\mu _{j})$, 
\begin{equation}
\label{G}
G(s, t):=\prod _{j=1}^{4}\Gamma \left( \frac{s+it +\mu _{j}}{2} \right) \Gamma \left( \frac{s-it +\overline{\mu _{j}}}{2} \right).
\end{equation}
Then the main theorem of this paper is as follows. 
\begin{thm}
Let $\pi =\otimes _{v}\pi _{v}$ be an irreducible cuspidal tempered automorphic representation of $\mathrm{GL}_{4}(\mathbb{A}_{\mathbb{Q}})$ with unitary central character.  Suppose that the  $L$-function $L(s, \pi)$ is given by (\ref{1}) and satisfies the functional equation (\ref{2}) for some positive integer $N_{\pi}$. Furthermore, we assume that there exists an absolute constant $K>0$ for which 
\begin{equation}
\label{assumption}
\sum _{\chi (\mathrm{mod}\; q)}\int _{S\leq |s| \leq 2S}\left| L(1/2+s, \pi \otimes \chi ) \right|^{2} ds \ll \varphi (q)N_{\pi}^{\varepsilon}(S+1)\log ^{K}( q(S+1))
\end{equation}
holds for any $S>0$, where the integral above is over some line segment parallel to the imaginary axis in $0\leq \Re (s)\leq c/\log q$ for some positive constant $c$. Let $\Psi (x)$ be a smooth function compactly supported in $[1,2]$. Then we have
\begin{equation}
\label{maintheorem}
\begin{aligned}
& \underset{(q, N_{\pi})=1}{\sum _{q}}\Psi (q/Q)\sum _{\chi (\mathrm{mod}\; q)}^{\quad \quad \flat}\int _{-\infty}^{\infty}\left| \Lambda (1/2+iy, \pi \otimes \chi ) \right|^{2} dy \\
& \quad =2 \pi ^{-\mathfrak{m}}\int _{-\infty}^{\infty}G(1/2, t) dt \underset{(q, N_{\pi})=1}{\sum _{q}}\varphi ^{\flat}(q)\Psi (q/Q)\frac{A(1/2)}{B_{q}(1/2)}\log q  +O(Q^{2}N_{\pi}^{47/64+\varepsilon }(\log Q)^{\varepsilon})
\end{aligned}
\end{equation}
uniformly for $Q\geq 2$, $N_{\pi}\ll (\log Q)^{N}$ for any fixed $N$. Here, $G(s,t)$ is given by (\ref{G}), and $A(s)$, $B_{q}(s)$ are given by 
\[
A(s)=\prod _{p}(1-p^{-2s})B_{p}(s), \quad B_{q}(s)=\prod _{p|q}B_{p}(s)
\]
respectively, where 
\[
B_{p}(s)=\int _{0}^{1}\prod _{j=1}^{4}(1-\alpha _{j}(p)e(\theta)p^{-s})^{-1}(1-\overline{\alpha _{j}(p)}e(-\theta)p^{-s})^{-1} d\theta 
\]
with $e(\theta ):=\exp (2\pi i \theta )$.  The implied constant in (\ref{maintheorem}) is independent of both $Q$ and $N_{\pi}$. 
\end{thm}
\begin{rem}
1. Since the main term of (\ref{maintheorem}) is of order $Q^{2}\log Q$, the main term dominates the error term if $N_{\pi}\ll (\log Q)^{64/47-\delta}$ for any $\delta >0$. \\
2. If $N_{\pi}=1$, then $B_{p}(1/2)$  is expressed by the Dirichlet coefficients of $L(s, \pi)$ by the following formula.  Define the polynomials $f_{i}=f_{i}(s_{1},s_{2},s_{3},s_{4})$ $(i=0, \ldots ,9)$ and $g_{i}=g_{i}(s_{1},s_{2},s_{3},s_{4})$ $(i=0, \ldots ,12)$ by 
\[
f_{0}=s_{4}^{2}, \quad f_{1}=3s_{4}^{2}, \quad f_{2}=-s_{2}^{2}s_{4}+6s_{4}^{2}, \quad
f_{3}=s_{2}s_{3}^{2}+s_{1}^{2}s_{2}s_{4}-3s_{2}^{2}s_{4}-2s_{1}s_{3}s_{4}+10s_{4}^{2},
\]
\[
f_{4}=-s_{1}^{2}s_{3}^{2}+2s_{2}s_{3}^{2}+2s_{1}^{2}s_{2}s_{4}-4s_{2}^{2}s_{4}-2s_{1}s_{3}s_{4}+12s_{4}^{2},
\]
\[
f_{i}=f_{9-i} \quad (i=5, \ldots ,9),
\]
\[g_{0}=s_{4}^{3}, \quad g_{1}=-s_{1}s_{3}s_{4}^{2}+4s_{4}^{3}, \quad 
g_{2}=s_{2}s_{3}^{2}s_{4}+s_{1}^{2}s_{2}s_{4}^{2}-2s_{2}^{2}s_{4}^{2}-4s_{1}s_{3}s_{4}^{2}+10s_{4}^{3},
\]
\begin{align*}
g_{3}=&-s_{3}^{4}-s_{1}s_{2}^{2}s_{3}s_{4}+7s_{2}s_{3}^{2}s_{4}-s_{1}^{4}s_{4}^{2}+7s_{1}^{2}s_{2}s_{4}^{2} -8s_{2}^{2}s_{4}^{2}-13s_{1}s_{3}s_{4}^{2}+20s_{4}^{3}, 
\end{align*}
\begin{align*}
g_{4}=&s_{1}s_{2}s_{3}^{3}-3s_{3}^{4}+s_{2}^{2}s_{4}+s_{1}^{3}s_{2}s_{3}s_{4}-8s_{1}s_{2}^{2}s_{3}s_{4}-s_{1}^{2}s_{3}^{2}s_{4}+18s_{2}s_{3}^{2}s_{4} \\
&-3s_{1}^{4}s_{4}^{2}+18s_{1}^{2}s_{2}s_{4}^{2}-16s_{2}^{2}s_{4}^{2}-24s_{1}s_{3}s_{4}^{2}+31s_{4}^{3}, 
\end{align*}
\begin{align*}
g_{5}=&-s_{2}^{3}s_{3}^{2}-s_{1}^{3}s_{3}^{3}+5s_{1}s_{2}s_{3}^{3}-6s_{3}^{4}-s_{1}^{2}s_{2}^{3}s_{4}+4s_{2}^{4}s_{4}+5s_{1}^{3}s_{2}s_{3}s_{4}-19s_{1}s_{2}^{2}s_{3}s_{4} \\
&-2s_{1}^{2}s_{3}^{2}s_{4}+29s_{2}s_{3}^{2}s_{4}-6s_{1}^{4}s_{4}^{2}+29s_{1}^{2}s_{2}s_{4}^{2}-24s_{2}^{2}s_{4}^{2}-34s_{1}s_{3}s_{4}^{2}+40s_{4}^{3}, 
\end{align*}
\begin{align*}
g_{6}=&s_{1}^{2}s_{2}^{2}s_{3}^{2}-2s_{2}^{3}s_{3}^{2}-21s_{1}^{3}s_{3}^{3}+6s_{1}s_{2}s_{3}^{3}-7s_{3}^{4}-2s_{1}^{2}s_{2}^{3}s_{4}+6s_{2}^{4}s_{4}+6s_{1}^{3}s_{2}s_{3}s_{4} \\
&-24s_{1}s_{2}^{2}s_{3}s_{4}+34s_{2}s_{3}^{2}s_{4}-7s_{1}^{4}s_{4}^{2}+34s_{1}^{2}s_{2}s_{4}^{2}-28s_{2}^{2}s_{4}^{2}-40s_{1}s_{3}s_{4}^{2}+44s_{4}^{3},
\end{align*}
\[
g_{i}=g_{12-i} \quad (i=7, \ldots ,12)
\]
and put 
\begin{equation}
\label{N}
\begin{aligned}
N_{\pi}(p)=&\sum _{i=0}^{9}f_{i}(a_{\pi}(p), a_{\pi}(p)^{2}-a_{\pi}(p^{2}),a_{\pi}(p)^{3}-2a_{\pi}(p)a_{\pi}(p^{2})+a_{\pi}(p^{3}),  \\
& \quad \quad \quad \quad a_{\pi}(p)^{4}-3a_{\pi}(p)^{2}a_{\pi}(p^{2})+a_{\pi}(p^{2})^{2}+2a_{\pi}(p)a_{\pi}(p^{3})-a_{\pi}(p^{4}))p^{i},
\end{aligned}
\end{equation}
\begin{equation}
\label{D}
\begin{aligned}
D_{\pi}(p)=&\sum _{i=0}^{12}g_{i}(a_{\pi}(p), a_{\pi}(p)^{2}-a_{\pi}(p^{2}),a_{\pi}(p)^{3}-2a_{\pi}(p)a_{\pi}(p^{2})+a_{\pi}(p^{3}),  \\
& \quad \quad \quad \quad a_{\pi}(p)^{4}-3a_{\pi}(p)^{2}a_{\pi}(p^{2})+a_{\pi}(p^{2})^{2}+2a_{\pi}(p)a_{\pi}(p^{3})-a_{\pi}(p^{4}))p^{i}.
\end{aligned}
\end{equation}
Then 
\begin{equation}
\label{Bp1}
B_{p}(1/2)=\frac{p^{4}(a_{\pi}(p)^{4}-3a_{\pi}(p)^{2}a_{\pi}(p^{2})+a_{\pi}(p^{2})^{2}+2a_{\pi}(p)a_{\pi}(p^{3})-a_{\pi}(p^{4}))}{p-1}\frac{N_{\pi}(p)}{D_{\pi}(p)}.
\end{equation}
In addition, $B_{p}(1/2)$ can be expressed by the Dirichlet coefficients of the logarithmic derivative of $L(s, \pi)$. Write 
\[
\frac{L^{'}(s, \pi)}{L(s, \pi)}=-\sum _{n=1}^{\infty}\frac{\tilde{a}_{\pi}(n)\Lambda (n)}{n^{s}},
\]
where $\Lambda$ is the von Mangoldt function. Then $N_{\pi}(p)$ and $D_{\pi}(p)$ above are expressed by 
\begin{equation}
\label{N2}
\begin{aligned}
N_{\pi}(p)=&\sum _{i=0}^{9}f_{i}\left( \tilde{a}_{\pi} (p), \frac{1}{2}(\tilde{a}_{\pi}(p)^{2}-\tilde{a}_{\pi}(p^{2})), \frac{1}{3}\tilde{a}_{\pi}(p^{3})+\frac{1}{6}\tilde{a}_{\pi}(p)^{3}-\frac{1}{2}\tilde{a}_{\pi}(p)\tilde{a}_{\pi}(p^{2}), \right.  \\
& \quad \quad \quad \quad \left. \frac{1}{24}\tilde{a}_{\pi}(p)^{4}-\frac{1}{4}\tilde{a}_{\pi}(p)^{2}\tilde{a}_{\pi}(p^{2})+\frac{1}{8}\tilde{a}_{\pi}(p^{2})^{2}+\frac{1}{3}\tilde{a}_{\pi}(p)\tilde{a}_{\pi}(p^{3})-\frac{1}{4}\tilde{a}_{\pi}(p^{4}) \right)p^{i},
\end{aligned}
\end{equation}
\begin{equation}
\label{D2}
\begin{aligned}
D_{\pi}(p)=&\sum _{i=0}^{12}g_{i}\left( \tilde{a}_{\pi} (p), \frac{1}{2}(\tilde{a}_{\pi}(p)^{2}-\tilde{a}_{\pi}(p^{2})), \frac{1}{3}\tilde{a}_{\pi}(p^{3})+\frac{1}{6}\tilde{a}_{\pi}(p)^{3}-\frac{1}{2}\tilde{a}_{\pi}(p)\tilde{a}_{\pi}(p^{2}), \right.  \\
& \quad \quad \quad \quad \left. \frac{1}{24}\tilde{a}_{\pi}(p)^{4}-\frac{1}{4}\tilde{a}_{\pi}(p)^{2}\tilde{a}_{\pi}(p^{2})+\frac{1}{8}\tilde{a}_{\pi}(p^{2})^{2}+\frac{1}{3}\tilde{a}_{\pi}(p)\tilde{a}_{\pi}(p^{3})-\frac{1}{4}\tilde{a}_{\pi}(p^{4}) \right)p^{i}
\end{aligned}
\end{equation}
and we have 
\begin{equation}
\label{Bp2}
B_{p}(1/2)=\frac{p^{4} (\frac{1}{24}\tilde{a}_{\pi}(p)^{4}-\frac{1}{4}\tilde{a}_{\pi}(p)^{2}\tilde{a}_{\pi}(p^{2})+\frac{1}{8}\tilde{a}_{\pi}(p^{2})^{2}+\frac{1}{3}\tilde{a}_{\pi}(p)\tilde{a}_{\pi}(p^{3})-\frac{1}{4}\tilde{a}_{\pi}(p^{4}) )  }{p-1}\frac{N_{\pi}(p)}{D_{\pi}(p)}.
\end{equation}
\end{rem}

\section{The approximation and decomposition of the moment}
Let   $\pi$ be an irreducible tempered cuspidal automorphic representation of $\mathrm{GL}_{4}(\mathbb{A}_{\mathbb{Q}})$ with unitary central character and conductor $N_{\pi}$,  and $\chi$ be an even primitive Dirichlet character modulo $q$, where $(q, N_{\pi})=1$.   Put
\begin{equation}
\label{8}
\Lambda (s, \pi \otimes \chi ;t):=\Lambda (s+it, \pi \otimes \chi )\Lambda (s-it, \overline{\pi \otimes \chi}),
\end{equation}
\begin{equation}
\label{9}
G_{\chi}(s,t):=\prod _{j=1}^{4}\Gamma \left( \frac{s+it+\mu _{j,\chi}}{2} \right) \Gamma \left( \frac{s-it+\overline{\mu _{j,\chi}}}{2} \right).
\end{equation}
Then we have the following formulas.
\begin{lem}
We have 
\begin{equation}
\label{13}
\Lambda (1/2, \pi \otimes \chi ;t)=2\pi ^{-\mathfrak{m}_{\chi}}\sum _{m,n=1}^{\infty}\frac{a_{\pi}(m)\overline{a_{\pi}(n)}\chi (m)\overline{\chi }(n)}{\sqrt{mn}}\left( \frac{n}{m} \right)^{it}W_{\chi}\left( \frac{\pi ^{4}mn}{q^{4}N_{\pi}},t \right)
\end{equation}
and
\begin{equation}
\label{14}
\int _{-\infty}^{\infty} \Lambda (1/2, \pi \otimes \chi ;t)dt =2\pi ^{-\mathfrak{m}_{\chi}}\sum _{m,n=1}^{\infty}\frac{a_{\pi}(m)\overline{a_{\pi}(n)}\chi (m)\overline{\chi }(n)}{\sqrt{mn}}V_{\chi}(m,n;qN_{\pi}^{1/4}),
\end{equation}
where $\mathfrak{m}_{\chi}:=\sum _{j=1}^{4}\Re (\mu _{j, \chi})$, 
\begin{equation}
\label{15}
W_{\chi}(x,t):=\frac{1}{2\pi i}\int _{(1)}G_{\chi}(1/2+s, t)x^{-s} \frac{ds}{s},
\end{equation}
\begin{equation}
\label{16}
\begin{aligned}
V_{\chi}(\xi ,\eta ;\mu )&:=\int _{-\infty}^{\infty}\left( \frac{\eta}{\xi} \right)^{it} W_{\chi}\left( \frac{\pi ^{4}\xi \eta}{\mu ^{4}}, t \right) dt \\
&=\frac{1}{2\pi i}\int _{-\infty}^{\infty}\int _{(1)}\left( \frac{\eta}{\xi} \right)^{it}G_{\chi}(1/2+s,t)\left( \frac{\pi ^{4}\xi \eta}{\mu ^{4}} \right)^{-s}\frac{ds}{s}dt.
\end{aligned}
\end{equation}
\end{lem}
\begin{proof}
Put 
\begin{align*}
P(\pi \otimes \chi ,t)&:=\frac{1}{2\pi i}\int _{(1)}\Lambda (1/2+s, \pi \otimes \chi ;t) \frac{ds}{s} \\
&=\frac{1}{2\pi i}\int _{(1)}\Lambda (1/2+s+it, \pi \otimes \chi)\Lambda (1/2+s-it, \overline{\pi \otimes \chi}) \frac{ds}{s}.
\end{align*}
We move the line of integration to $\Re (s)=-1$. Then we cross a single pole at $s=0$ of order $1$ and the residue is $\Lambda (1/2, \pi \otimes \chi ;t)$. By the functional equation (\ref{6}), we see that the new integral equals $-P(\pi \otimes \chi ,t)$. Hence we have
\[
\Lambda (1/2, \pi \otimes \chi ;t)=2P(\pi \otimes \chi ,t).
\]
Expanding $L$-functions as Dirichlet  series, we have 
\[
P(\pi \otimes \chi ,t)=\pi ^{-\mathfrak{m}_{\chi}}\sum _{m,n=1}^{\infty}\frac{a_{\pi}(m)\overline{a_{\pi}(n)}\chi (m)\overline{\chi}(n)}{\sqrt{mn}}\left( \frac{n}{m} \right)^{it}\frac{1}{2\pi i}\int _{(1)}G_{\chi}(1/2+s, t)\left( \frac{\pi ^{4}mn}{q^{4}N_{\pi}} \right)^{-s} \frac{ds}{s}.
\]
Thus we get (\ref{13}). The identity (\ref{14}) is obtained by integrating both sides of (\ref{13}) by $t$. 
\end{proof}

It is known that if $(q, N_{\pi})=1$, even primitive Dirichlet characters $\chi$ modulo $q$ satisfy $\{\mu _{j, \chi}\}_{j=1}^{4}= \{\mu _{j}\}_{j=1}^{4}$ (see \cite{RS}, for example), so the  parameter $\mathfrak{m}_{\chi}$ and the function $G_{\chi}$ (hence $W_{\chi}$, $V_{\chi}$) are all independent of $\chi$. Hence from here we denote these parameter and functions by $\mathfrak{m}$, $G$, $W$ and $V$ respectively. Then it follows from the above lemma that
\begin{equation}
\label{18}
\begin{aligned}
& \underset{(q, N_{\pi})=1}{\sum _{q}}\Psi (q/Q)\sum _{\chi (\mathrm{mod}\; q)}^{\quad \quad \flat}\int  _{-\infty}^{\infty}\left| \Lambda (1/2+iy, \pi \otimes \chi ) \right|^{2} dy \\
& \quad =2\pi ^{-\mathfrak{m}}\sum _{m,n=1}^{\infty}\frac{a_{\pi}(m)\overline{a_{\pi}(n)}}{\sqrt{mn}}\underset{(q, N_{\pi})=1}{\sum _{q}}\Psi (q/Q)V(m, n; qN_{\pi}^{1/4})\sum _{\chi (\mathrm{mod}\; q)}^{\quad \quad \flat}\chi(m)\overline{\chi}(n).
\end{aligned}
\end{equation}
To handle the sum over even primitive characters, we use the following lemma.
\begin{lem}[\cite{Sound}, (2.1)]
If $(mn,q)=1$, then
\[
\sum _{\chi (\mathrm{mod}\; q)}^{\quad \quad \flat}\chi (m)\overline{\chi}(n)=\frac{1}{2}\underset{r|(m\pm n)}{\sum _{q=dr}}\mu (d)\varphi (r),
\]
where we sum over both choices of sign.
\end{lem}
By this lemma, the right hand side of (\ref{18}) becomes 
\begin{equation}
\label{19}
\pi ^{-\mathfrak{m}} \sum _{m,n=1}^{\infty}\frac{a_{\pi}(m)\overline{a_{\pi}(n)}}{\sqrt{mn}} \underset{(dr, mn N_{\pi})=1}{\underset{r|(m\pm n)}{\sum _{d,r}} }\mu (d)\varphi (r)\Psi (dr/Q) V(m,n ;dr N_{\pi}^{1/4}).
\end{equation}
We put 
\begin{equation}
\label{20}
\Delta _{\pi}(\Psi ,Q):=\frac{1}{2}  \sum _{m,n=1}^{\infty}\frac{a_{\pi}(m)\overline{a_{\pi}(n)}}{\sqrt{mn}} \underset{(dr, mn N_{\pi})=1}{\underset{r|(m\pm n)}{\sum _{d,r}} }\mu (d)\varphi (r)\Psi (dr/Q) V(m,n ;dr N_{\pi}^{1/4}).
\end{equation}
Then 
\begin{equation}
\label{21}
\underset{(q, N_{\pi})=1}{\sum _{q}}\Psi (q/Q)\sum _{\chi (\mathrm{mod}\; q)}^{\quad \quad \flat}\int  _{-\infty}^{\infty}\left| \Lambda (1/2+iy, \pi \otimes \chi ) \right|^{2} dy=2\pi ^{-\mathfrak{m}}\Delta _{\pi}(\Psi ,Q).
\end{equation}
We approximate $\Delta _{\pi}(\Psi ,Q)$  by 
\begin{equation}
\label{22}
\tilde{\Delta}_{\pi}(\Psi ,Q):=\underset{(q, N_{\pi})=1}{\sum _{q}}\Psi (q/Q)\sum _{\chi (\mathrm{mod}\; q)}^{\quad \quad \flat}\sum _{m,n=1}^{\infty}\frac{a_{\pi}(m)\overline{a_{\pi}(n)}} {\sqrt{mn}}\chi (m)\overline{\chi}(n)V\left(m,n; \frac{qN_{\pi}^{1/4}}{(\log Q)^{\alpha}} \right),
\end{equation}
where $\alpha >0$ is a proper  constant which will be  chosen precisely later. To estimate the difference between $\Delta _{\pi}$ and $\tilde{\Delta }_{\pi}$, we adapt the following large sieve inequality. 
\begin{lem}
For any complex numbers $a_{n}$ $(M\leq n <M+N)$, we have
\[
\sum _{q\leq Q}\frac{q}{\varphi (q)}\sum _{\chi (\mathrm{mod}\; q)}^{\quad \quad *}\left| \sum _{M\leq n <M+N}a_{n}\chi (n) \right|^{2} \leq (Q^{2}+N)\sum _{M\leq n <M+N}|a_{n}|^{2}.
\]
\end{lem}
The following lemma gives a bound for the  second moment of $a_{\pi }(n)$.
\begin{lem}
Let $m\geq 2$ be a positive integer and $\pi $ be a  cuspidal automorphic representation of $\textrm{GL}_{m}(\mathbb{A}_{\mathbb{Q}})$ with unitary central character. Let  $a_{\pi}(n)$ be the $n$th Dirichlet coefficient of  $L(s, \pi)$ and ${\cal Q}$ be its analytic conductor.  Then we have
\begin{equation}
\label{D4}
\sum _{x \leq n <2x}|a_{\pi}(n)|^{2}\ll _{\varepsilon ,m}{\cal Q}^{\varepsilon}x +{\cal Q}^{3/2m +\varepsilon}x^{1-1/m^{2}}.
\end{equation}
\end{lem}

\begin{proof}
The proof of this lemma relies deeply on the results in \cite{BTZ}. Let ${\cal A}(d)$ be the set of cuspidal automorphic representations of $\textrm{GL}_{d}(\mathbb{A}_{\mathbb{Q}})$ with unitary central character and put ${\cal A}=\cup _{d\geq 1}{\cal A}(d)$. For $\pi \in {\cal A}$, we denote by ${\cal Q}(\pi)$ the analytic conductor of $\pi$. For ${\cal Q}>0$, put 
\[
{\cal A}_{m}({\cal Q}):=\{ \pi \in {\cal A}\; \vline \; {\cal Q}(\pi )\leq {\cal Q},\; \pi \in {\cal A}(d)\Rightarrow d \leq m \}.
\] 
(In \cite{BTZ}, a bit more general set ${\cal F}_{m}({\cal Q})$ is defined instead of ${\cal A}_{m}({\cal Q})$. The set ${\cal F}_{m}({\cal Q})$ is made from a subset ${\cal F}\subset {\cal A}$, and ${\cal A}_{m}({\cal Q})$ is obtained by taking ${\cal F}={\cal A}$.) Let $\phi$ be a smooth function compactly supported in $(-2,2)$. Then by Lemma 5.5 (with $d=T=1$, $\pi ^{'}=\pi$) of \cite{BTZ}, for $\pi \in {\cal A}_{m}({\cal Q})$, we have
\begin{align*}
\underset{(n, N_{\pi})=1}{\sum _{n=1}^{\infty}}|a_{\pi}(n)|^{2}\phi (\log n/x)=&g_{1}^{RS}(1, \pi \times \tilde{\pi})H(1, \pi \times \tilde{\pi})x\hat{\phi}(1)\prod _{p|N_{\pi}}L^{RS}(1, \pi _{p}\times \tilde{\pi}_{p})^{-1} \\
& \quad +O_{\varepsilon ,m, \phi}({\cal Q}^{3/2m}x^{1-1/m^{2}}),
\end{align*}
where the coefficients in the main term satisfy 
\[
g_{1}^{RS}(1, \pi \times \tilde{\pi})\ll 1, \quad H(1, \pi \times \tilde{\pi})\ll _{\varepsilon ,m}{\cal Q}^{\varepsilon}, \quad \prod _{p|N_{\pi}}L^{RS}(1, \pi _{p}\times \tilde{\pi}_{p})^{-1}\ll {\cal Q}^{\varepsilon}.
\]
(See Lemma 5.3, Lemma 5.2 and page 20 of \cite{BTZ} respectively.) Hence we have an upper bound 
\[
\underset{(n, N_{\pi})=1}{\sum _{n=1}^{\infty}}|a_{\pi}(n)|^{2}\phi (\log n/x)\ll _{\varepsilon ,m,\phi}{\cal Q}^{\varepsilon}x +{\cal Q}^{3/2m}x^{1-1/m^{2}}.
\]
We choose a nonnegative function $\phi$ satisfying $\phi (t) \geq 1$ for $0\leq t \leq \log 2$. Then the above estimate yields 
\[
\underset{x\leq n <2x}{\underset{(n, N_{\pi})=1}{\sum _{n=1}^{\infty}}}|a_{\pi}(n)|^{2}\ll _{\varepsilon ,m}{\cal Q}^{\varepsilon}x +{\cal Q}^{3/2m}x^{1-1/m^{2}}
\]
for $\pi \in {\cal A}_{m}({\cal Q})$. We need to remove the condition $(n, N_{\pi})=1$. By the above estimate, we have
\begin{align*}
\underset{(n, N_{\pi})>1}{\sum _{x \leq n<2x}}|a_{\pi}(n)|^{2}& \leq  \underset{l|N_{\pi}}{\sum _{l>1}}     \underset{x\leq ln^{'}<2x} {\underset{(n^{'},N_{\pi})=1}{\sum _{n^{'}=1}^{\infty}}}    |a_{\pi}(ln^{'})|^{2}      \\
&=  \underset{l|N_{\pi}}{\sum _{l>1}}|a_{\pi}(l)|^{2}     \underset{x/l\leq n^{'}<2x/l} {\underset{(n^{'},N_{\pi})=1}{\sum _{n^{'}=1}^{\infty}}}    |a_{\pi}(n^{'})|^{2}      \\
& \ll  \underset{l|N_{\pi}}{\sum _{l>1}} |a_{\pi}(l)|^{2} \left( {\cal Q}^{\varepsilon}x/l+{\cal Q}^{3/2m}(x/l)^{1-1/m^{2}} \right).
\end{align*}
The Luo-Rudnick-Sarnak bound (see \cite{LRS}) yields $|a_{\pi}(l)|\ll l^{1/2-1/(m^{2}+1)+\varepsilon}$. Then 
\begin{align*}
 \underset{l|N_{\pi}}{\sum _{l>1}}\frac{|a_{\pi}(l)|^{2}}{l^{1-1/m^{2}}} &\ll _{ \varepsilon ,m}  \underset{l|N_{\pi}}{\sum _{l>1}}l^{-\frac{m^{2}-1}{m^{2}(m^{2}+1)}+\varepsilon}  \leq d(N_{\pi}) \ll _{\varepsilon} N_{\pi}^{\varepsilon}\ll {\cal Q}^{\varepsilon}.
\end{align*}
Hence we also have
\[
 \underset{l|N_{\pi}}{\sum _{l>1}}\frac{|a_{\pi}(l)|^{2}}{l} \ll _{ \varepsilon ,m} {\cal Q}^{\varepsilon}.
\]
Therefore, we have 
\[
\underset{(n, N_{\pi})>1}{\sum _{x \leq n <2x}}|a_{\pi}(n)|^{2}\ll _{\varepsilon ,m} {\cal Q}^{\varepsilon}x +{\cal Q}^{3/2m +\varepsilon}x^{1-1/m^{2}}.
\] 
Combining these estimates we obtain (\ref{D4}). 
\end{proof}
We shall use Lemma 2.4 with ${\cal Q}$ replaced by $N_{\pi}$, since we do not take the size of $\mu_{j}$ in the functional equation into account. By our definitions, 
\begin{equation}
\label{25}
\begin{aligned}
&\Delta _{\pi}(\Psi ,Q)-\tilde{\Delta}_{\pi}(\Psi ,Q) \\
&= \underset{(q, N_{\pi})=1}{\sum _{q}}\Psi (q/Q)\sum _{\chi (\mathrm{mod}\; q)}^{\quad \quad \flat}\sum _{m,n=1}^{\infty}\frac{a_{\pi}(m)\overline{a_{\pi}(n)}\chi (m)\overline{\chi}(n)} {\sqrt{mn}}  \left\{ V(m,n; qN_{\pi}^{1/4})-V \left(m,n; \frac{qN_{\pi}^{1/4}}{(\log Q)^{\alpha}}   \right)   \right\},
\end{aligned}
\end{equation}
and 
\begin{equation}
\label{26}
\begin{aligned}
&  V(m,n; qN_{\pi}^{1/4})-V \left(m,n; \frac{qN_{\pi}^{1/4}}{(\log Q)^{\alpha}}   \right) \\
&\quad =\frac{1}{2\pi i}\int _{-\infty}^{\infty}\left( \frac{n}{m} \right)^{it}\int _{(1)}G(1/2+s,t)\left\{\left( \frac{\pi ^{4}mn}{q^{4}N_{\pi}} \right)^{-s}-\left( \frac{\pi ^{4}mn(\log Q)^{4\alpha}}{q^{4}N_{\pi}} \right)^{-s}   \right\} \frac{ds}{s}dt.
\end{aligned}
\end{equation}
Changing the parameters by $s=(v+z)/2$, $t=(v-z)/(2i)$, the right hand side of (\ref{26}) becomes 
\begin{equation}
\label{27}
\begin{aligned}
& \frac{1}{2\pi}\int _{(1)}\int _{(1)}\left( \frac{n}{m} \right)^{\frac{v-z}{2}}G\left( \frac{1+v+z}{2}, \frac{v-z}{2i} \right) \frac{1}{(mn)^{\frac{v+z}{2}}} \\
& \hspace{3cm} \times \left\{ \left( \frac{qN_{\pi}^{1/4}}{\pi} \right)^{2(v+z)}-\left( \frac{qN_{\pi}^{1/4}}{\pi (\log Q)^{\alpha}} \right)^{2(v+z)}    \right\} \frac{dvdz}{v+z}.
\end{aligned}
\end{equation}
Substituting (\ref{26}), (\ref{27}) into (\ref{25}), we have 
\begin{equation}
\label{28}
\begin{aligned}
& \Delta _{\pi}(\Psi ,Q)-\tilde{\Delta}_{\pi}(\Psi ,Q) \\
& \quad = \frac{1}{2\pi}  \underset{(q, N_{\pi})=1}{\sum _{q}}\Psi (q/Q)\sum _{\chi (\mathrm{mod}\; q)}^{\quad \quad \flat}\sum _{m,n=1}^{\infty} \int _{(1)}\int _{(1)}\frac{a_{\pi}(m)\overline{a_{\pi}(n)}\chi (m)\overline{\chi}(n)}{m^{1/2+v}n^{1/2+z}} \\
& \quad \quad \quad  \times G\left( \frac{1+v+z}{2}, \frac{v-z}{2i} \right)  \left\{ \left( \frac{qN_{\pi}^{1/4}}{\pi} \right)^{2(v+z)}-\left( \frac{qN_{\pi}^{1/4}}{\pi (\log Q)^{\alpha}} \right)^{2(v+z)}    \right\} \frac{dvdz}{v+z}.
\end{aligned}
\end{equation}
Let $\sum ^{d}$ be the dyadic sum and $F_{M}$ be a positive smooth function supported in $[M/2, 3M]$ satisfying $F_{M}^{(j)}(x)\ll _{j}M^{-j}$ $(j \geq 0)$, $1=\sum _{M}^{d}F_{M}(x)$. Then 
\begin{equation}
\label{29}
\begin{aligned}
&  \Delta _{\pi}(\Psi ,Q)-\tilde{\Delta}_{\pi}(\Psi ,Q) \\
& \quad = \frac{1}{2\pi}  \sum _{M}^{\quad \quad d}\sum _{N}^{\quad \quad d}\int _{(1)}\int _{(1)}G\left( \frac{1+v+z}{2}, \frac{v-z}{2i} \right)  \left\{ \left( \frac{QN_{\pi}^{1/4}}{\pi} \right)^{2(v+z)}-\left( \frac{QN_{\pi}^{1/4}}{\pi (\log Q)^{\alpha}} \right)^{2(v+z)}    \right\}  \\
& \quad \times \underset{(q, N_{\pi})=1}{\sum _{q}}\Psi (q/Q) (q/Q)^{2v+2z}\sum _{\chi (\mathrm{mod}\; q)}^{\quad \quad \flat}\sum _{m=1}^{\infty}\frac{a_{\pi}(m)\chi (m)}{m^{1/2+v}}F_{M}(m)\sum _{n=1}^{\infty}\frac{\overline{a_{\pi}(n)}\overline{\chi}(n)}{n^{1/2+z}}F_{N}(n) \frac{dvdz}{v+z}.
\end{aligned}
\end{equation}
By Lemmas 2.3-2.4 with Cauchy-Schwarz inequality, we have 
\begin{align*}
&  \underset{(q, N_{\pi})=1}{\sum _{q}}\Psi (q/Q)\sum _{\chi (\mathrm{mod}\; q)}^{\quad \quad \flat}\sum _{m=1}^{\infty}\frac{a_{\pi}(m)\chi (m)}{m^{1/2+v}}F_{M}(m)\sum _{n=1}^{\infty}\frac{\overline{a_{\pi}(n)}\overline{\chi}(n)}{n^{1/2+z}}F_{N}(n)  \\
& \quad \leq \left(  \sum _{q}\Psi (q/Q)\sum _{\chi (\mathrm{mod}\; q)}^{\quad \quad \flat}\left| \sum _{m=1}^{\infty}\frac{a_{\pi}(m)\chi (m)}{m^{1/2+v}}F_{M}(m) \right|^{2} \right)^{1/2} \\
& \quad \quad \quad  \times \left(  \sum _{q}\Psi (q/Q)\sum _{\chi (\mathrm{mod}\; q)}^{\quad \quad \flat}\left| \sum _{n=1}^{\infty}\frac{\overline{a_{\pi}(n)}\overline{\chi} (n)}{n^{1/2+z}}F_{N}(n) \right|^{2} \right)^{1/2} \\
& \quad  \ll (Q^{2}+M)^{1/2}(Q^{2}+N)^{1/2}\left( \sum _{M/2 \leq m \leq 3M}\frac{|a_{\pi}(m)|^{2}}{m^{1+2\Re (v)}} \right)^{1/2} \left( \sum _{N/2 \leq n \leq 3N}\frac{|a_{\pi}(n)|^{2}}{n^{1+2\Re (z)}} \right)^{1/2}  \\
& \quad \ll N_{\pi}^{\varepsilon}(Q^{2}+M)^{1/2}(Q^{2}+N)^{1/2}M^{-\Re (v)}N^{-\Re (z)}(1+N_{\pi}^{3/8}M^{-1/16})^{1/2} (1+N_{\pi}^{3/8}N^{-1/16})^{1/2}. 
\end{align*}
We use this to estimate the integral in (\ref{29}). It might be helpful to keep in mind that the function $G((1+v+z)/2, (v-z)/2i)$ in (\ref{29}) does not have any pole in the domain 
\begin{equation}
\label{Ghol}
\Re \left( \frac{1+v+z}{2} \right)- \left| \Re \left( \frac{v-z}{2} \right) \right| >0.
\end{equation}
We consider the following cases. \\
{\bf Case I}. $M>Q^{2}$. (By symmetry, the case $N>Q^{2}$ is the same.) We do not shift the integral over $v$. The integration over $z$ is shifted to $\Re (z)=l$, where $l=-1/4$ if $N\leq Q^{2}$ and otherwise $l=15/32 +\varepsilon$. We do not encounter the poles of $G$. The integral in (\ref{29}) is bounded by 
\begin{align*}
& \int _{(l)}\int _{(1)}\left| G\left( \frac{1+v+z}{2}, \frac{v-z}{2i} \right) \right|(QN_{\pi}^{1/4})^{2l+2} \\
& \quad \times N_{\pi}^{\varepsilon}(Q^{2}+M)^{1/2}(Q^{2}+N)^{1/2}M^{-1}N^{-l}(1+N_{\pi}^{3/8}M^{-1/16})^{1/2} (1+N_{\pi}^{3/8}N^{-1/16})^{1/2}|dv||dz| \\
& \quad \ll \frac{Q^{2+2l}N_{\pi}^{\frac{l+1}{2}+\varepsilon}}{MN^{l}}M^{1/2}(Q^{2}+N)^{1/2}(1+N_{\pi}^{3/8}M^{-1/16})^{1/2} (1+N_{\pi}^{3/8}N^{-1/16})^{1/2}.
\end{align*}
Hence the contribution of this to (\ref{29}) is at most 
\begin{equation}
\label{delta1}
\begin{aligned}
& \sum _{M>Q^{2}}^{\quad \quad d}\sum _{N\leq Q^{2}}^{\quad \quad d}\frac{Q^{3/2}N_{\pi}^{\frac{3}{8}+\varepsilon}}{M}M^{1/2}N^{1/4}(Q^{2}+N)^{1/2}(1+N_{\pi}^{3/8}M^{-1/16})^{1/2}(1+N_{\pi}^{3/8}N^{-1/16})^{1/2} \\
&+\sum _{M>Q^{2}}^{\quad \quad d}\sum _{N>Q^{2}}^{\quad \quad d}\frac{Q^{47/16}N_{\pi}^{47/64 +\varepsilon}}{M^{1/2}}N^{1/32-\varepsilon}(1+N_{\pi}^{3/8}M^{-1/16})^{1/2}(1+N_{\pi}^{3/8}N^{-1/16})^{1/2} \\
& \quad \ll Q^{2}N_{\pi}^{47/64 +\varepsilon}(1+N_{\pi}^{3/8}Q^{-1/8}).
\end{aligned}
\end{equation}

\noindent
{\bf Case II}. $Q^{2}/(\log Q)^{2\alpha}<M\leq Q^{2}$, $N\leq Q^{2}$. (By symmetry, the case $M\leq Q^{2}$, $Q^{2}/(\log Q)^{2\alpha}<N\leq Q^{2}$ is the same.)\\
We shift the lines of integration to $\Re (v)=0$, $\Re (z)=l$, where $l=0$ if $Q^{2}/(\log Q)^{2\alpha}<N\leq Q^{2}$ and otherwise $l=-1/4$. Then we do not encounter the poles of the integrand. The integral in (\ref{29}) is at most
\begin{align*}
& \int _{(l)}\int _{(0)}\left| G \left( \frac{1+v+z}{2}, \frac{v-z}{2i} \right) \right| (QN_{\pi}^{1/4})^{2l}(\log Q)^{\varepsilon} \\
& \quad \times Q^{2}N^{-l}N_{\pi}^{\varepsilon}(1+N_{\pi}^{3/8}M^{-1/16})^{1/2} (1+N_{\pi}^{3/8}N^{-1/16})^{1/2} |dv| |dz| \\
& \quad \ll \frac{Q^{2+2l}N_{\pi}^{l/2+\varepsilon}(\log Q)^{\varepsilon}}{N^{l}}(1+N_{\pi}^{3/8}M^{-1/16})^{1/2} (1+N_{\pi}^{3/8}N^{-1/16})^{1/2}.
\end{align*}
Hence the contribution of this to (\ref{29}) is at most 
\begin{align*}
& \sum _{\frac{Q^{2}}{(\log Q)^{2\alpha}}<M\leq Q^{2}}^{\quad \quad d} \sum _{\frac{Q^{2}}{(\log Q)^{2\alpha}}<N\leq Q^{2}}^{\quad \quad d} Q^{2}(\log Q)^{\varepsilon}N_{\pi}^{\varepsilon} (1+N_{\pi}^{3/8}M^{-1/16})^{1/2} (1+N_{\pi}^{3/8}N^{-1/16})^{1/2} \\
& \quad + \sum _{\frac{Q^{2}}{(\log Q)^{2\alpha}}<M\leq Q^{2}}^{\quad \quad d}\sum _{1\leq N \leq \frac{Q^{2}}{(\log Q)^{2\alpha}}}^{\quad \quad d}Q^{3/2}(\log Q)^{\varepsilon}N^{1/4}N_{\pi}^{-1/8+\varepsilon} (1+N_{\pi}^{3/8}M^{-1/16})^{1/2} (1+N_{\pi}^{3/8}N^{-1/16})^{1/2} \\
& \quad \ll Q^{2}N_{\pi}^{3/8+\varepsilon}(\log Q)^{\varepsilon}.
\end{align*}
\noindent
{\bf Case III}. $M \leq Q^{2}/(\log Q)^{2\alpha}$, $N\leq Q^{2}/(\log Q)^{2\alpha}$. We shift the paths of integration in (\ref{29}) to $\Re (v)=\Re (z)=-1/4$. Then we do not cross the poles of the integrand. The new integral is at most 
\begin{align*}
& \int _{(-1/4)}\int _{(-1/4)}\left| G \left( \frac{1+v+z}{2}, \frac{v-z}{2i} \right) \right|(QN_{\pi}^{1/4})^{-1}(\log Q)^{\alpha} \\
& \quad \quad \times Q^{2}M^{1/4}N^{1/4}N_{\pi}^{\varepsilon} (1+N_{\pi}^{3/8}M^{-1/16})^{1/2} (1+N_{\pi}^{3/8}N^{-1/16})^{1/2} |dv| |dz| \\
& \ll Q(\log Q)^{\alpha} M^{1/4}N^{1/4}N_{\pi}^{-1/4+\varepsilon}(1+N_{\pi}^{3/8}M^{-1/16})^{1/2} (1+N_{\pi}^{3/8}N^{-1/16})^{1/2}.
\end{align*}
Hence the contribution of this part to (\ref{29}) is at most 
\begin{align*}
& \sum _{M\leq \frac{Q^{2}}{(\log Q)^{2\alpha}}}^{\quad \quad d}\sum _{N\leq \frac{Q^{2}}{(\log Q)^{2\alpha}}}^{\quad \quad d}Q(\log Q)^{\alpha} M^{1/4}N^{1/4}N_{\pi}^{-1/4+\varepsilon}(1+N_{\pi}^{3/8}M^{-1/16})^{1/2}(1+N_{\pi}^{3/8}N^{-1/16})^{1/2} \\
& \quad   \ll Q^{2}N_{\pi}^{1/8 +\varepsilon}. 
\end{align*}
Summing up, we have the following conclusion.
\begin{prop}
We have 
\[
\Delta _{\pi}(\Psi ,Q)-\tilde{\Delta}_{\pi}(\Psi ,Q)\ll Q^{2}(\log Q)^{\varepsilon} N_{\pi}^{47/64 +\varepsilon}(1+N_{\pi}^{3/8}Q^{-1/8}).
\]
\end{prop}

We estimate
\begin{equation}
\label{32}
\tilde{\Delta}_{\pi}(\Psi ,Q)=\frac{1}{2}\sum _{m,n=1}^{\infty}\frac{a_{\pi}(m)\overline{a_{\pi}(n)}}{\sqrt{mn}}\underset{(dr, mn N_{\pi})=1}{\underset{r|(m \pm n)}{\sum _{d,r}}     }\mu (d)\varphi (r)\Psi (dr/Q)V\left( m,n;\frac{drN_{\pi}^{1/4}}{(\log Q)^{\alpha}} \right).
\end{equation}
For some constant $\delta >0$, put 
\[
D=(\log Q)^{\delta}.
\]
We decompose $\tilde{\Delta}_{\pi}(\Psi ,Q)$ by 
\begin{equation}
\label{33}
\tilde{\Delta}_{\pi}(\Psi ,Q)={\cal D}_{\pi}(\Psi ,Q)+{\cal S}_{\pi}(\Psi ,Q)+{\cal G}_{\pi}(\Psi ,Q),
\end{equation}
where ${\cal D}_{\pi}$ denotes  the terms with $m=n$, ${\cal S}_{\pi}$ denotes  the terms with $m\neq n$, $d\leq D$ and ${\cal G}_{\pi}$ denotes the remaining terms.


\section{The computation of ${\cal D}_{\pi}(\Psi ,Q)$}
By the definition of $V(\xi, \eta ; \mu)$, we have
\begin{equation}
\label{34}
\begin{aligned}
& {\cal D}_{\pi}(\Psi ,Q) \\
& \quad =\underset{(q, N_{\pi})=1}{\sum _{q}}\Psi (q/Q)\varphi ^{\flat}(q)\underset{(n,q)=1}{\sum _{n=1}^{\infty}}\frac{|a_{\pi}(n)|^{2}}{n}V\left(n,n; \frac{qN_{\pi}^{1/4}}{(\log Q)^{\alpha}} \right) \\
& \quad =\underset{(q, N_{\pi})=1}{\sum _{q}}\varphi ^{\flat}(q)\Psi (q/Q)\frac{1}{2\pi i}\int _{-\infty}^{\infty}\int _{(1)}\left( \underset{(n,q)=1}{\sum _{n=1}^{\infty}}\frac{|a_{\pi}(n)|^{2}}{n^{1+2s}} \right)G(1/2+s, t)\left( \frac{qN_{\pi}^{1/4}}{\pi (\log Q)^{\alpha}} \right)^{4s}\frac{ds}{s}dt.
\end{aligned}
\end{equation}
The sum over $n$  in (\ref{34}) can be computed by using the recipe in \cite{CFKRS}, p.53-60. Define the sequence $(a_{n})$ by $a_{n}:=a_{\pi}(n)$ if $(n, q)=1$ and otherwise $a_{n}:=0$. Then the function ${\cal L}_{p}(x)=\sum _{n=0}^{\infty}a_{p^{n}}x^{n}$ in \cite{CFKRS} is given by 
\[
{\cal L}_{p}(x)=\begin{cases}
\prod _{j=1}^{4}(1-\alpha _{j}(p)x)^{-1} &(p\nmid q) \\
1 &(p\; |\; q),
\end{cases}
\]
where $\alpha _{j}(p)$ is the Satake parameter in the Euler product of $L(s, \pi)$. By the argument in \cite{CFKRS}, we obtain the following formula. 
\begin{lem}[\cite{CFKRS}, Theorem 2.4.1]
Let $\delta$ be an arbitrarily fixed positive number. For $\Re (s)>-1/4+\delta$, we have 
\begin{equation}
\label{35}
\underset{(n,q)=1}{\sum _{n=1}^{\infty}}\frac{|a_{\pi}(n)|^{2}}{n^{1+2s}}=\frac{\zeta (1+2s)A(1/2+s)}{B_{q}(1/2+s)},
\end{equation}
where 
\begin{equation}
\label{A}
A(s)=\prod _{p}(1-p^{-2s})B_{p}(s),
\end{equation}
\begin{equation}
\label{Bq}
B_{q}(s)=\prod _{p|q}B_{p}(s)
\end{equation}
with
\begin{equation}
\label{Bp}
B_{p}(s)=\int _{0}^{1}\prod _{j=1}^{4}(1-\alpha _{j}(p)e(\theta)p^{-s})^{-1}(1-\overline{\alpha _{j}(p)}e(-\theta)p^{-s})^{-1}d\theta .
\end{equation}
\end{lem}
Substituting (\ref{35}) into (\ref{34}), we obtain 
\begin{equation}
\label{36}
\begin{aligned}
& {\cal D}_{\pi}(\Psi ,Q) \\
& \quad =\underset{(q, N_{\pi})=1}{\sum _{q}}\varphi ^{\flat}(q)\Psi (q/Q)\frac{1}{2\pi i}\int _{-\infty}^{\infty}\int _{(1)}\zeta (1+2s)\frac{A(1/2+s)}{B_{q}(1/2+s)}G(1/2+s, t)\left( \frac{qN_{\pi}^{1/4}}{\pi (\log Q)^{\alpha}} \right)^{4s}\frac{ds}{s}dt.
\end{aligned}
\end{equation}
We shift the contour of the $s$-integral to $\Re (s)=-1/4+\delta$ $(\delta >0)$. Then we cross the pole at $s=0$ of order 2. The residue at $s=0$ is 
\begin{equation}
\label{37}
\begin{aligned}
& \frac{d}{ds}\left(s^{2} \zeta (1+2s)\frac{A(1/2+s)}{B_{q}(1/2+s)}G(1/2+s, t)\left( \frac{qN_{\pi}^{1/4}}{\pi (\log Q)^{\alpha}} \right)^{4s}\frac{1}{s}  \right)\vline _{s=0} \\
& \quad =2\frac{A(1/2)}{B_{q}(1/2)}G(1/2, t)\log q +O\left(\left(|G(1/2, t)|+|\frac{\partial G}{\partial s}(1/2, t)|\right)(\log N_{\pi})(\log Q)^{\varepsilon}   \right)
\end{aligned}
\end{equation}
for $q \asymp Q$. The new integral is small enough.  By (\ref{36}) and (\ref{37}) we obtain the following conclusion. 
\begin{prop}
We have
\begin{equation}
\label{38}
{\cal D}_{\pi}(\Psi ,Q)=2 \int _{-\infty}^{\infty}G(1/2,t) dt \underset{(q,N_{\pi})=1}{\sum _{q}}\varphi ^{\flat}(q)\Psi (q/Q)\frac{A(1/2)}{B_{q}(1/2)}\log q +O(Q^{2}(\log N_{\pi})(\log Q)^{\varepsilon}),
\end{equation}
where $A(s)$, $B_{q}(s)$ are defined by (\ref{A}), (\ref{Bq}) respectively.
\end{prop}
We assume $|\alpha _{j}(p)|=1$ for $j=1, \ldots ,4$ and  compute 
\[
B_{p}(1/2)=\int _{0}^{1}\prod _{j=1}^{4}(1-\alpha _{j}(p)e(\theta)p^{-1/2})^{-1} \prod _{j=1}^{4}(1-\overline{\alpha _{j}(p)}e(-\theta)p^{-1/2})^{-1} d\theta .
\]
Put $z=e(\theta)=e^{2\pi i \theta}$. Then $z$ moves on the unit circle $C=\{z=e^{2\pi i \theta}|0\leq \theta \leq 2\pi \}$ and 
\begin{align*}
B_{p}(1/2)=&\int _{C}\prod _{j=1}^{4}(1-\alpha _{j}(p)p^{-1/2}z)^{-1} \prod _{j=1}^{4}(1-\overline{\alpha _{j}(p)}p^{-1/2}z^{-1})^{-1} \frac{dz}{2\pi iz} \\
=&\frac{1}{2\pi i}\int _{C}\frac{ z^{3} }{ \prod _{j=1}^{4}(1-\alpha _{j}(p)p^{-1/2}z) \prod _{j=1}^{4}(z-\overline{\alpha _{j}(p)}p^{-1/2}) } dz.
\end{align*}
To compute the integration above, we temporarily assume that $\alpha _{i}(p)\neq \alpha _{j}(p)$ whenever $i\neq j$.  (Due to the continuity, we may forget this assumption after the residual computation.)  Since the integrand has totally 4 poles of order $1$ at $z=\overline{\alpha _{k}(p)}p^{-1/2}$ $(k=1, \ldots ,4)$, we have 
\begin{align*}
B_{p}(1/2)=& \sum _{k=1}^{4} \mathrm{Res} _{z=\overline{\alpha _{k}(p)}p^{-1/2}  }\left(\frac{z^{3}}{\prod _{j=1}^{4}(1-\alpha _{j}(p)p^{-1/2}z) \prod _{j=1}^{4}(z-\overline{\alpha _{j}(p)}p^{-1/2})   }    \right) \\
=&  \sum _{k=1}^{4} \frac{\overline{\alpha _{k}(p)}^{3}}{\prod _{j=1}^{4}(1-\alpha _{j}(p)\overline{\alpha _{k}(p)}p^{-1})\underset{j\neq k}{\prod _{j=1}^{4}}(\overline{\alpha _{k}(p)}-\overline{\alpha _{j}(p)})} \\
=& p^{4}\sum _{k=1}^{4}\frac{\alpha _{1}(p) \alpha _{2}(p) \alpha _{3}(p) \alpha _{4}(p) \alpha _{k}(p)^{3}}{\prod _{j=1}^{4}(p\alpha _{k}(p)-\alpha _{j}(p))\underset{j\neq k}{\prod _{j=1}^{4}}(\alpha _{j}(p)-\alpha _{k}(p))   }.
\end{align*}
We write this by 
\[
B_{p}(1/2)=\frac{p^{4} \alpha _{1}(p) \alpha _{2}(p) \alpha _{3}(p) \alpha _{4}(p) }{p-1}\frac{N_{\pi}(p)}{D_{\pi}(p)},
\]
where 
\[
D_{\pi}(p):= \underset{j\neq k}{\prod _{j,k=1}^{4}}(p\alpha _{k}(p)-\alpha _{j}(p)).
\]
Let
\[
s_{1}=\alpha _{1}(p)+\ldots +\alpha _{4}(p), \quad s_{2}=\alpha _{1}(p)\alpha _{2}(p)+\ldots +\alpha _{3}(p)\alpha _{4}(p),
\]
\[
s_{3}=\alpha _{1}(p)\alpha _{2}(p)\alpha _{3}(p)+\ldots +\alpha _{2}(p)\alpha _{3}(p)\alpha _{4}(p), \quad s_{4}=\alpha _{1}(p)\alpha _{2}(p)\alpha _{3}(p)\alpha _{4}(p)
\]
be the basic symmetric polynomials of $\alpha _{1}(p), \ldots ,\alpha _{4}(p)$. Then by numerical computation, we see that  $N_{\pi}(p)$ and $D_{\pi}(p)$ have the expressions
\[
N_{\pi}(p)=\sum _{i=0}^{9}f_{i}(s_{1},s_{2},s_{3},s_{4})p^{i}, \quad D_{\pi}(p)=\sum _{i=0}^{12}g_{i}(s_{1},s_{2},s_{3},s_{4})p^{i},
\]
where $f_{i}$ and $g_{i}$ are given in Remark 1.2. Put 
\[
p_{r}:=a_{\pi}(p^{r})=\sum _{k_{1}+k_{2}+k_{3}+k_{4}=r}\alpha _{1}(p)^{k_{1}}\alpha _{2}(p)^{k_{2}}\alpha _{3}(p)^{k_{3}}\alpha _{4}(p)^{k_{4}}.
\]
Then $s_{1}, \ldots ,s_{4}$ are expressed by 
\[
s_{1}=p_{1}, \quad s_{2}=p_{1}^{2}-p_{2},
\]
\[
s_{3}=p_{1}^{3}-2p_{1}p_{2}+p_{3}, \quad s_{4}=p_{1}^{4}-3p_{1}^{2}p_{2}+p_{2}^{2}+2p_{1}p_{3}-p_{4}.
\]
Hence we obtain (\ref{Bp1}). On the other hand, put 
\[
q_{r}:=\tilde{a}_{\pi}(p^{r})=\sum _{j=1}^{4}\alpha _{j}(p)^{r}.
\]
Then 
\[
s_{1}=q_{1}, \quad s_{2}=\frac{1}{2}(q_{1}^{2}-q_{2}),
\]
\[
s_{3}=\frac{1}{3}q_{3}+\frac{1}{6}q_{1}^{3}-\frac{1}{2}q_{1}q_{2}, \quad 
s_{4}=\frac{1}{24}q_{1}^{4}-\frac{1}{4}q_{1}^{2}q_{2}+\frac{1}{8}q_{2}^{2}+\frac{1}{3}q_{1}q_{3}-\frac{1}{4}q_{4}.
\]
Hence we obtain (\ref{Bp2}).


\section{Estimation of ${\cal S}_{\pi}(\Psi ,Q)$}
The term ${\cal S}_{\pi}(\Psi ,Q)$ is defined by 
\[
{\cal S}_{\pi}(\Psi ,Q)=\frac{1}{2}\underset{m\neq n}{\sum _{m,n=1}^{\infty}}\frac{a_{\pi}(m)\overline{a_{\pi}(n)}}{\sqrt{mn}} \underset{(dr, mn N_{\pi})=1}{ \underset{d>D, r|(m\pm n)}{\sum _{d,r}}} \mu (d)\varphi (r)\Psi (dr/Q)V\left( m,n; \frac{drN_{\pi}^{1/4}}{(\log Q)^{\alpha}} \right).
\]
We reintroduce the terms  $m=n$. This process gives an error term
\begin{align*}
& \sum _{n=1}^{\infty}\frac{|a_{\pi}(n)|^{2}}{n}\underset{d>D}{\underset{(dr,nN_{\pi})=1}{\sum _{d,r}}   }\mu (d)\varphi (r)\Psi (dr/Q)V\left( n,n; \frac{drN_{\pi}^{1/4}}{(\log Q)^{\alpha}} \right) \\
& \quad  =\underset{d>D}{\underset{(dr,N_{\pi})=1}{\sum _{d,r}}   }\mu (d)\varphi (r)\Psi (dr/Q) \\
& \quad \quad \times \frac{1}{2\pi i}\int _{-\infty}^{\infty}\int _{(1)}\underset{(n,dr)=1}{\sum _{n=1}^{\infty}}\frac{|a_{\pi}(n)|^{2}}{n^{1+2s}}G(1/2+s,t)\left( \frac{drN_{\pi}^{1/4}}{\pi (\log Q)^{\alpha}} \right)^{4s}\frac{ds}{s}dt.
\end{align*}
By the computation of ${\cal D}_{\pi}(\Psi ,Q)$, we see that
\[
\frac{1}{2\pi i} \int _{(1)}\underset{(n,dr)=1}{\sum _{n=1}^{\infty}}\frac{|a_{\pi}(n)|^{2}}{n^{1+2s}}G(1/2+s,t)\left( \frac{drN_{\pi}^{1/4}}{\pi (\log Q)^{\alpha}} \right)^{4s}\frac{ds}{s} \ll \log Q\log N_{\pi}
\]
for $dr \asymp Q$. Hence the error term is at most
\begin{equation}
\label{39}
\log Q\log N_{\pi}\underset{d>D}{\sum _{d,r}}\varphi (r)\Psi (dr/Q) \ll \log Q\sum _{d>D}\sum _{r \leq 2Q/d}r \ll \frac{Q^{2}\log Q\log N_{\pi}}{D}.
\end{equation}
For simplicity we skip to write the error term bounded by (\ref{39}).  Now we have 
\[
{\cal S}_{\pi}(\Psi ,Q)=\frac{1}{2}\underset{ }{\sum _{m,n=1}^{\infty}}\frac{a_{\pi}(m)\overline{a_{\pi}(n)}}{\sqrt{mn}} \underset{(dr, mn N_{\pi})=1}{ \underset{d>D, r|(m\pm n)}{\sum _{d,r}}} \mu (d)\varphi (r)\Psi (dr/Q)V\left( m,n; \frac{drN_{\pi}^{1/4}}{(\log Q)^{\alpha}} \right).
\]
We replace the condition $r|(m\pm n)$ with the sum over even Dirichlet characters modulo $r$. Then 
\[
{\cal S}_{\pi}(\Psi ,Q)=\underset{d>D}{\underset{(dr,N_{\pi})=1}{\sum _{d,r}}  }\mu (d)\Psi (dr/Q)\underset{\chi (-1)=1}{\sum _{\chi (\mathrm{mod}\; r)}  }\underset{(mn,d)=1}{\sum _{m,n=1}^{\infty}}  \frac{a_{\pi}(m)\overline{a_{\pi}(n)}\chi (m)\overline{\chi}(n)}{\sqrt{mn}}V\left( m,n; \frac{drN_{\pi}^{1/4}}{(\log Q)^{\alpha}} \right).
\]
By the definition of $V(\xi ,\eta ;\mu)$, we have 
\begin{align*}
{\cal S}_{\pi}(\Psi ,Q)=& \underset{d>D}{\underset{(dr,N_{\pi})=1}{\sum _{d,r}}} \mu (d)\Psi (dr/Q) \frac{1}{2\pi i}\int _{-\infty}^{\infty}\int _{(1)}\underset{\chi (-1)=1}{\sum _{\chi (\mathrm{mod}\; r)}  }\underset{(mn,d)=1}{\sum _{m,n=1}^{\infty}} \frac{a_{\pi}(m)\overline{a_{\pi}(n)}\chi (m)\overline{\chi}(n)}{m^{1/2+s+it}n^{1/2+s-it}} \\
& \quad \quad \times G(1/2+s, t)\left( \frac{\pi ^{4}(\log Q)^{4\alpha}}{d^{4}r^{4}N_{\pi}} \right)^{-s}\frac{ds}{s}dt.
\end{align*}
Writing the sum over $m$, $n$ as a product of  $L$-functions, we have 
\begin{equation}
\label{42}
\begin{aligned}
{\cal S}_{\pi}(\Psi ,Q)=& \underset{d>D}{\underset{(dr,N_{\pi})=1}{\sum _{d,r}}} \mu (d)\Psi (dr/Q) \frac{1}{2\pi i}\int _{-\infty}^{\infty}\int _{(1)}\underset{\chi (-1)=1}{\sum _{\chi (\mathrm{mod}\; r)}  } \frac{L(1/2+s+it, \pi \otimes \chi)L(1/2+s-it, \tilde{\pi}\otimes \overline{\chi})}{L_{d}(1/2+s+it, \pi \otimes \chi)L_{d}(1/2+s-it, \tilde{\pi}\otimes \overline{\chi})   }     \\
& \quad \quad \times G(1/2+s, t)\left( \frac{\pi ^{4}(\log Q)^{4\alpha}}{d^{4}r^{4}N_{\pi}} \right)^{-s}\frac{ds}{s}dt,
\end{aligned}
\end{equation}
where 
\[
L_{d}(s, \pi \otimes \chi)=\prod _{p|d}\prod _{j=1}^{4}\left(1-\frac{\alpha _{j}(p)\chi (p)}{p^{s}} \right)^{-1}.
\]
For $\Re (s)>1/2$, 
\[
L_{d}(s, \pi \otimes \chi )^{-1}\ll \prod _{p|d}\prod _{j=1}^{4}\left( 1+\frac{|\alpha _{j}(p)|}{p^{1/2}} \right) \ll 2^{4\omega (d)}\ll d^{\varepsilon}.
\]
We shift the line of integration to $\Re (s)=1/\log Q$. Put $s=1/\log Q +iv$, $t_{1}=v+t$, $t_{2}=v-t$. Furthermore, we decompose the sum over even characters by $\sum _{\chi (\textrm{mod}\; r)}=\sum _{l|r}\sum _{\chi ^{'}(\textrm{mod}\; l)}^{\flat}$. If a Dirichlet character $\chi$ ($\mathrm{mod}\; r$) is induced by a primitive character $\chi^{'}$ ($\mathrm{mod}\; l$), then 
\begin{align*}
L(1/2 +1/\log Q +it, \pi \otimes \chi) & =L(1/2 +1/\log Q +it, \pi \otimes \chi ^{'})\underset{p\nmid l}{\prod _{p|r}}\prod _{j=1}^{4}\left( 1-\frac{\alpha _{j}(p)\chi ^{'}(p)}{p^{1/2+1/\log Q +it}} \right) \\
&  \ll |L(1/2 +1/\log Q +it, \pi \otimes \chi ^{'})|\prod _{p|\frac{r}{l}}2^{4} \\
&  \ll  |L(1/2 +1/\log Q +it, \pi \otimes \chi ^{'})|\prod _{p|\frac{r}{l}}\tau (r/l)^{4}.
\end{align*}
Furthermore, we have
\[
G(1/2+s, t)\ll \exp (-c(|t_{1}|+|t_{2}|))
\]
for some constant $c>0$. Combining these estimates with $|ab|\leq a^{2}+b^{2}$, we have
\begin{equation}
\label{43}
\begin{aligned}
& {\cal S}_{\pi}(\Psi ,Q) \\
& \quad \ll \log Q \sum _{d>D}d^{\varepsilon}\int _{-\infty}^{\infty}\int _{-\infty}^{\infty}\exp (-c(|t_{1}|+|t_{2}|))  \sum _{l\leq \frac{2Q}{d}}\underset{l|r}{\sum _{r}}\tau (r/l)^{8} \\
& \quad \quad \times\sum _{\chi ^{'}(\mathrm{mod}\; l)}^{\quad \quad \flat}\left\{ |L(1/2+1/\log Q+it_{1}, \pi \otimes \chi)|^{2}+ |L(1/2+1/\log Q+it_{2}, \tilde{\pi} \otimes \overline{\chi})|^{2} \right\} dt_{1}dt_{2}.
\end{aligned}
\end{equation}
Our assumption (\ref{assumption}) in Theorem 1.1 yields 
\begin{align*}
& \int _{-\infty}^{\infty}\int _{-\infty}^{\infty}\exp (-c(|t_{1}|+|t_{2}|))  \\
& \quad \times \sum _{\chi ^{'}(\mathrm{mod}\; l)}^{\quad \quad \flat}\left\{ |L(1/2+1/\log Q+it_{1}, \pi \otimes \chi)|^{2}+ |L(1/2+1/\log Q+it_{2}, \tilde{\pi} \otimes \overline{\chi})|^{2} \right\} dt_{1}dt_{2} \\
& \quad \ll N_{\pi}^{\varepsilon}\varphi (l)(\log l)^{K}.
\end{align*}
Hence by (\ref{43}), we have 
\begin{align*}
 {\cal S}_{\pi}(\Psi ,Q) &\ll N_{\pi}^{\varepsilon }\log Q\sum _{d>D}d^{\varepsilon}\sum _{l \leq \frac{2Q}{d}} l(\log l)^{K}       \underset{l|r}{\sum _{r}}\tau (r/l)^{8}\\ 
 &\ll N_{\pi}^{\varepsilon }\log Q\sum _{d>D}d^{\varepsilon}\sum _{l \leq \frac{2Q}{d}} l(\log l)^{K}       \frac{2Q}{dl}(\log Q)^{A^{'}} \\
 &\ll N_{\pi}^{\varepsilon }Q(\log Q)^{A^{'}+1}\sum _{d>D}\frac{1}{d^{1-\varepsilon}}\sum _{l \leq \frac{2Q}{d}}(\log l)^{K} \\
 &\ll \frac{N_{\pi}^{\varepsilon}Q^{2}(\log Q)^{A^{'}+K+1}}{D^{1-\varepsilon}},
\end{align*}
where $A^{'}>0$ is some absolute constant.  The right hand side is larger than that of (\ref{39}). Thus we have the following conclusion. 
\begin{prop}
We have
\begin{equation}
\label{S}
{\cal S}_{\pi}(\Psi ,Q)\ll  \frac{N_{\pi}^{\varepsilon}Q^{2}(\log Q)^{A}}{D^{1-\varepsilon}},
\end{equation}
where  $A$ is some positive constant dependent only on $K$ in  (\ref{assumption}) and the implied constant is dependent only on $\varepsilon >0$. 
\end{prop}


\section{Estimation of ${\cal G}_{\pi}(\Psi ,Q)$}
In this section, we estimate 
\begin{equation}
\label{45}
{\cal G}_{\pi}(\Psi ,Q)=\frac{1}{2}\underset{m\neq n}{\sum _{m,n=1}^{\infty}}\frac{a_{\pi}(m)\overline{a_{\pi}(n)}}{\sqrt{mn}} \underset{(dr, mn N_{\pi})=1}{ \underset{d\leq D, r|(m\pm n)}{\sum _{d,r}}} \mu (d)\varphi (r)\Psi (dr/Q)V\left( m,n; \frac{drN_{\pi}^{1/4}}{(\log Q)^{\alpha}} \right).
\end{equation}
Write $g=(m,n)$, $m=gM$, $n=gN$ with $(M,N)=1$. Moreover, we replace $\varphi (r)$ with $\sum _{al=r}\mu (a)l$. Since $r|m\pm n$, $(r,mn)=1$ is equivalent to $(r,g)=1$. Hence the sum over $d, r$ in (\ref{45}) becomes
\begin{equation}
\label{46}
\underset{d \leq D,\; al|M\pm N}{\underset{(d, mnN_{\pi})=1,\; (al, g)=1}{\sum _{d,a,l}}  }\mu (d)\mu (a)l \Psi( dal/Q)V\left(m,n;\frac{dalN_{\pi}^{1/4}}{(\log Q)^{\alpha}}  \right).
\end{equation}
Write $|M\pm N|=alh$. Then $l$ is replaced with $|M\pm N|/ah$. The condition $(l,g)=1$ is removed by multiplying $\sum _{b|(l,g)}\mu (b)$. Write $l=bk$. Then the sum  (\ref{46}) becomes 
\[
\underset{(d, gN_{\pi}MN)=1}{\sum _{d \leq D}}\mu (d)\underset{(a,g)=1}{\sum _{a}}\mu (a)\underset{b|g}{\sum _{b}}\mu (b)\underset{|M\pm N|=abkh}{\sum _{k\geq 1}}bk \Psi (dabk/Q)V\left( gM,gN; \frac{dabkN_{\pi}^{1/4}}{(\log Q)^{\alpha}} \right).
\]
Substituting $k=|M\pm N|/abh$ and writing the sum above as the sum over $d,a,b,h$, then  (\ref{46}) equals
\begin{equation}
\label{47}
\begin{aligned}
& Q\underset{(d,gN_{\pi}MN)=1}{\sum _{d\leq D}}\underset{(a,g)=1}{\sum _{a}}\underset{b|g}{\sum _{b}}\underset{M\equiv \mp N (\mathrm{mod}\; abh)}{\sum _{h>0}}\frac{\mu (d)\mu (a)\mu (b)}{ad}\frac{d|M\pm N|}{Qh} \\
& \quad \quad \quad \quad \quad  \times \Psi \left(\frac{d|M\pm N|}{Qh}  \right)V\left( gM,gN; \frac{d|M\pm N|N_{\pi}^{1/4}}{h(\log Q)^{\alpha}} \right).
\end{aligned}
\end{equation} 
For $u, x, y \in \mathbb{R}_{\geq 0}$, put 
\[
{\cal W}^{\pm}(x,y;u)=u|x \pm y|\Psi (u|x\pm y|)V(x,y; u|x\pm y|).
\]
Since $V(\xi ,\eta ;\mu)$ satisfies 
\[
V(cm, cn ;\sqrt{c}\mu)=V(m,n; \mu) 
\]
for any $c>0$, by taking $c=Q^{2}N_{\pi}^{1/2}/(\log Q)^{2\alpha}$, we have 
\[
V\left( gM,gN; \frac{d|M\pm N|N_{\pi}^{1/4}}{h(\log Q)^{\alpha}} \right)=V\left( \frac{gM(\log Q)^{2\alpha}}{Q^{2}N_{\pi}^{1/2}}, \frac{gN(\log Q)^{2\alpha}}{Q^{2}N_{\pi}^{1/2}}; \frac{d|M\pm N|}{Qh} \right).
\]
Hence
\begin{align*}
& \frac{d|M\pm N|}{Qh}\Psi \left(\frac{d|M\pm N|}{Qh}  \right)V\left( gM,gN; \frac{d|M\pm N|N_{\pi}^{1/4}}{h(\log Q)^{\alpha}} \right) \\
& \quad ={\cal W}^{\pm}\left(\frac{gM(\log Q)^{2\alpha}}{Q^{2}N_{\pi}^{1/2}},\frac{gN(\log Q)^{2\alpha}}{Q^{2}N_{\pi}^{1/2}}; \frac{QN_{\pi}^{1/2}d}{gh(\log Q)^{2\alpha}}   \right).
\end{align*}
Therefore, (\ref{47}) is rewritten as 
\begin{equation}
\label{48}
\begin{aligned}
& Q\underset{(d,gN_{\pi}MN)=1}{\sum _{d\leq D}}\underset{(a,g)=1}{\sum _{a}}\underset{b|g}{\sum _{b}}\underset{M\equiv \mp N (\mathrm{mod}\; abh)}{\sum _{h>0}}\frac{\mu (d)\mu (a)\mu (b)}{ad} \\
& \quad \quad \quad \quad \quad  \times {\cal W}^{\pm}\left(\frac{gM(\log Q)^{2\alpha}}{Q^{2}N_{\pi}^{1/2}},\frac{gN(\log Q)^{2\alpha}}{Q^{2}N_{\pi}^{1/2}}; \frac{QN_{\pi}^{1/2}d}{gh(\log Q)^{2\alpha}}   \right).
\end{aligned}
\end{equation}
Substituting this into (\ref{45}), we obtain 
\begin{equation}
\label{49}
\begin{aligned}
& {\cal G}_{\pi}(\Psi ,Q) \\
& \quad =\frac{Q}{2}\underset{m\neq n}{\sum _{m,n=1}^{\infty}}\frac{a_{\pi}(m)\overline{a_{\pi}(n)}}{\sqrt{mn}}\underset{(d,gN_{\pi}MN)=1}{\sum _{d\leq D}}  \underset{(a,g)=1}{\sum _{a}}\underset{b|g}{\sum _{b}}  \underset{M\equiv \mp N (\mathrm{mod}\; abh)}{\sum _{h>0}}\frac{\mu (d)\mu (a)\mu (b)}{ad} \\
& \quad \quad \quad \quad \quad  \times {\cal W}^{\pm}\left(\frac{gM(\log Q)^{2\alpha}}{Q^{2}N_{\pi}^{1/2}},\frac{gN(\log Q)^{2\alpha}}{Q^{2}N_{\pi}^{1/2}}; \frac{QN_{\pi}^{1/2}d}{gh(\log Q)^{2\alpha}}   \right).
\end{aligned}
\end{equation}
Suppose $a>2Q$. Since $M\neq N$, $M\equiv \mp N (\textrm{mod}\; abh)$, it follows that $|M\pm N|\geq abh$. Hence
\[
\frac{QN_{\pi}^{1/2}d}{gh(\log Q)^{2\alpha}} \left|\frac{gM(\log Q)^{2\alpha}}{Q^{2}N_{\pi}^{1/2}}\pm \frac{gN(\log Q)^{2\alpha}}{Q^{2}N_{\pi}^{1/2}}  \right|=\frac{d|M\pm N|}{Qh}\geq \frac{dab}{Q}>2.
\]
Since $\Psi (x)$ is supported in $[1,2]$, in this case we have
\[
{\cal W}^{\pm}\left(\frac{gM(\log Q)^{2\alpha}}{Q^{2}N_{\pi}^{1/2}},\frac{gN(\log Q)^{2\alpha}}{Q^{2}N_{\pi}^{1/2}}; \frac{QN_{\pi}^{1/2}d}{gh(\log Q)^{2\alpha}}   \right)=0.
\]
Therefore, we may restrict the sum to $a\leq 2Q$. Thus we have
\begin{equation}
\label{50}
\begin{aligned}
& {\cal G}_{\pi}(\Psi ,Q) \\
& \quad =\frac{Q}{2}\underset{m\neq n}{\sum _{m,n=1}^{\infty}}\frac{a_{\pi}(m)\overline{a_{\pi}(n)}}{\sqrt{mn}}\underset{(d,gN_{\pi}MN)=1}{\sum _{d\leq D}}\underset{(a,g)=1}{\sum _{a\leq 2Q}}\underset{b|g}{\sum _{b}}\underset{M\equiv \mp N (\mathrm{mod}\; abh)}{\sum _{h>0}}\frac{\mu (d)\mu (a)\mu (b)}{ad} \\
& \quad \quad \quad \quad \quad  \times {\cal W}^{\pm}\left(\frac{gM(\log Q)^{2\alpha}}{Q^{2}N_{\pi}^{1/2}},\frac{gN(\log Q)^{2\alpha}}{Q^{2}N_{\pi}^{1/2}}; \frac{QN_{\pi}^{1/2}d}{gh(\log Q)^{2\alpha}}   \right).
\end{aligned}
\end{equation}
To estimate (\ref{50}), we apply the following lemma.
\begin{lem}[\cite{CL}, Lemma 6.2]
For $s_{1},s_{2}\in \mathbb{C}$, $u>0$, put
\[
\tilde{\cal W}^{\pm}(s_{1},s_{2};u)=\int _{0}^{\infty}\int _{0}^{\infty}{\cal W}^{\pm}(x,y;u)x^{s_{1}}y^{s_{2}}\frac{dx}{x}\frac{dy}{y}.
\]
Then  the functions $\tilde{\cal W}^{\pm}(s_{1},s_{2};u)$ are analytic in the domain $\Re (s_{1}), \Re (s_{2})>0$ and the inversion formula
\begin{equation}
\label{51}
{\cal W}^{\pm}(x,y;u)=\frac{1}{(2\pi i)^{2}}\int _{(c_{1})}\int _{(c_{2})}\tilde{\cal W}^{\pm}(s_{1},s_{2};u)x^{-s_{1}}y^{-s_{2}}ds_{1}ds_{2}
\end{equation}
holds for any $c_{1},c_{2}>0$. Moreover, for any positive integer $k$, there exists a constant $c=c_{k}>0$ such that
\begin{equation}
\label{52}
|{\cal W}^{\pm}(s_{1},s_{2};u)|\ll \frac{(1+u)^{k-1}}{\max \{|s_{1}|,|s_{2}| \}^{k}}\exp (-cu^{-1/4})
\end{equation}
holds. 
\end{lem}
We use (\ref{51}) to replace ${\cal W}^{\pm}$ with the integral of $\tilde{\cal W}^{\pm}$. We also replace the condition $M \equiv \mp N (\textrm{mod}\; abh)$ with the sum over Dirichlet  characters modulo $abh$. Then we have
\begin{equation}
\label{53}
\begin{aligned}
& {\cal G}_{\pi}(\Psi ,Q) \\
& \quad =\frac{Q}{2}\sum _{a\leq 2Q}\sum _{b>0}\sum _{h>0}\sum _{\chi (\mathrm{mod}\; abh)}\underset{b|g, (a,g)=1}{\sum _{g}}\underset{(d,gN_{\pi})=1}{\sum _{d\leq D}}\frac{\mu (d)\mu (a)\mu (b)}{adg \varphi (abh)} \\
& \quad \quad \times \frac{1}{(2\pi i)^{2}}\int _{(\frac{1}{2}+\varepsilon)}\int _{(\frac{1}{2}+\varepsilon)}\tilde{\cal W}^{\pm}\left(s_{1},s_{2}; \frac{QN_{\pi}^{1/2}d}{gh(\log Q)^{2\alpha}}  \right) \left(  \frac{Q^{2}N_{\pi}^{1/2}}{g(\log Q)^{2\alpha}}\right)^{s_{1}+s_{2}} \\
& \quad \quad  \times \underset{(MN,d)=1}{ \underset{M\neq N, (M,N)=1}{\sum _{M,N=1}^{\infty}}    }\frac{a_{\pi}(gM)\overline{a_{\pi}(gN)}\chi (M)\overline{\chi}(\mp N)}{M^{1/2+s_{1}}N^{1/2+s_{2}}}ds_{1}ds_{2}.
\end{aligned}
\end{equation}
We estimate the series 
\[
{\cal L}^{\pm}(s_{1},s_{2};d,g;\chi):=\underset{(MN,d)=1}{ \underset{M\neq N, (M,N)=1}{\sum _{M,N=1}^{\infty}}    }\frac{a_{\pi}(gM)\overline{a_{\pi}(gN)}\chi (M)\overline{\chi}(\mp N)}{M^{1/2+s_{1}}N^{1/2+s_{2}}}.
\]
\begin{lem}
For $\Re (s_{1}), \Re (s_{2})>1/2$, we have 
\begin{equation}
\label{L}
\begin{aligned}
& {\cal L}^{\pm}(s_{1},s_{2};d,g;\chi) \\
& \quad =\overline{\chi}(\mp 1) \left( L(s_{1}+1/2, \pi \otimes \chi)L(s_{2}+1/2, \tilde{\pi}\otimes \overline{\chi})\lambda (s_{1},s_{2};d,g;\chi)\theta (s_{1},s_{2};d,g;\chi)-|a_{\pi}(g)|^{2} \right), \\
\end{aligned}
\end{equation}
where the functons $\lambda (s_{1},s_{2};d,g;\chi)$,  $\theta (s_{1},s_{2};d,g;\chi)$ are continued holomorphically to the domain $\Re (s_{1}), \Re (s_{2})>0$ and satisfy the bounds 
\[
\lambda (s_{1},s_{2};d,g;\chi) \ll \log Q,
\]
\[
\theta (s_{1},s_{2};d,g;\chi) \ll \tau (d)^{c_{1}}\tau (g)^{c_{2}}
\]
uniformly on the lines $\Re (s_{i})=1/\log Q$ $(i=1,2)$, where $c_{1}, c_{2}$ are some absolute constants. 
\end{lem}
\begin{proof}
Since $(M,N)=1$, the condition $M=N$ implies $M=N=1$. Hence 
\begin{equation}
\label{A-}
\begin{aligned}
{\cal L}^{\pm}(s_{1},s_{2};d,g;\chi)  =& \left ( \underset{(MN,d)=1}{ \underset{ (M,N)=1}{\sum _{M,N=1}^{\infty}}    }\frac{a_{\pi}(gM)\overline{a_{\pi}(gN)}\chi (M)\overline{\chi}( N)}{M^{1/2+s_{1}}N^{1/2+s_{2}}} -|a_{\pi}(g)|^{2} \right) \overline{\chi}(\mp 1) \\
=:& \left({\cal L}_{1}(s_{1},s_{2}; d,g ;\chi)-|a_{\pi}(g)|^{2} \right) \overline{\chi}(\mp 1),
\end{aligned}
\end{equation}
say. We write 
\begin{equation}
\label{B}
{\cal L}_{1}(s_{1},s_{2};d,g;\chi)=L(s_{1}+1/2, \pi \otimes \chi)L(s_{2}+1/2, \tilde{\pi}\otimes \overline{\chi}){\cal L}_{2}(s_{1},s_{2};d,g;\chi)
\end{equation}
and decompose ${\cal L}_{2}$ by 
\begin{equation}
\label{C}
{\cal L}_{2}(s_{1},s_{2};d,g;\chi)=\prod _{p}{\cal M}_{p}(s_{1},s_{2};d,g;\chi).
\end{equation}
Since $\pi _{p}$ is tempered for any $p$, the Satake parameters satisfy the Ramanujan bound $|\alpha _{j}(p)|\leq 1$ for $j=1, \ldots ,4$. Hence the Dirichlet coefficient
\[
a_{\pi}(p^{r})=\sum _{r_{1}+ \ldots +r_{4}=r}\alpha _{1}(p)^{r_{1}}\cdots \alpha _{4}(p)^{r_{4}}
\]
satisfies 
\[
|a_{\pi}(p^{r})|\leq \sum _{r_{1}+ \ldots +r_{4}=r}1 \leq (r+1)^{4}.
\]
We denote the $p$-factor of $L(s, \pi \otimes \chi)$ by $L_{p}(s, \pi \otimes \chi)$.  \\
{\bf Case I}. Suppose $p\nmid dg$. Then 
\begin{align*}
& {\cal M}_{p}(s_{1},s_{2};d,g;\chi) \\
& \quad =L_{p}(s_{1}+1/2, \pi \otimes \chi)^{-1}L_{p}(s_{2}+1/2, \tilde{\pi}\otimes \overline{\chi})^{-1} \left(1+\sum _{r=1}^{\infty}\frac{a_{\pi}(p^{r})\chi (p^{r})}{p^{(1/2+s_{1})r}}+\sum _{r=1}^{\infty}\frac{\overline{a_{\pi}(p^{r})}\overline{\chi }(p^{r})}{p^{(1/2+s_{2})r}}     \right) \\
& \quad =1+O\left(\frac{1}{p^{2\Re (s_{1})+1}} +\frac{1}{p^{2\Re (s_{2})+1}}+\frac{1}{p^{\Re (s_{1})+\Re (s_{2})+1}} \right).
\end{align*}
Put
\[
\lambda (s_{1},s_{2};d,g;\chi):=\prod _{p\nmid dg}{\cal M}_{p}(s_{1},s_{2};d,g;\chi).
\]
Then by the above computation, $\lambda (s_{1},s_{2};d,g;\chi)$ is holomorphic in the domain $\Re (s_{1}), \Re (s_{2})>0$ and on the set $\Re (s_{1})=\Re (s_{2})=100/\log Q$, we have
\begin{equation}
\label{d}
\lambda (s_{1},s_{2};d,g;\chi)\ll \log Q.
\end{equation}
{\bf Case II}. Suppose $p|d$. Then 
\begin{align*}
&  {\cal M}_{p}(s_{1},s_{2};d,g;\chi) \\
& \quad =L_{p}(s_{1}+1/2, \pi \otimes \chi)^{-1}L_{p}(s_{2}+1/2, \tilde{\pi}\otimes \overline{\chi})^{-1} \\
& \quad =\prod _{j=1}^{4}\left( 1-\frac{\alpha _{j}(p)\chi (p)}{p^{s_{1}+1/2}} \right)\prod _{j=1}^{4}\left( 1-\frac{\overline{\alpha _{j}(p)\chi (p)}}{p^{s_{2}+1/2}} \right) \\
& \quad =\left( 1-\frac{a_{\pi}(p)\chi (p)}{p^{s_{1}+1/2}} +O\left( \frac{1}{p^{1+2\Re (s_{1})}} \right) \right) \left( 1-\frac{\overline{a_{\pi}(p)\chi (p)}}{p^{s_{2}+1/2}} +O\left( \frac{1}{p^{1+2\Re (s_{2})}} \right) \right).
\end{align*}
Therefore, $\prod _{p|d}{\cal M}_{p}(s_{1},s_{2};d,g;\chi)$ is entire and on the set $\Re (s_{1})=\Re (s_{2})=100/\log Q$, we have
\begin{equation}
\label{515}
\begin{aligned}
\prod _{p|d}{\cal M}_{p}(s_{1},s_{2};d,g;\chi)\ll \prod _{p|d}\left( 1+\frac{2}{p^{1/2+100/\log Q}}+\frac{c}{p^{1+200/\log Q}}  \right) \leq \prod _{p|d}(3+c)^{2}  \leq \tau (d)^{c_{1}},
\end{aligned}
\end{equation}
where $c$ is some positive constant and $c_{1}=2\log (3+c)/\log 2$. \\
{\bf Case III}. Suppose $p|g$. Write $g_{p}=p^{\nu}$, where $\nu =\nu _{p,g}$ satisfies $p^{\nu}||g$. Then 
\begin{align*}
&   {\cal M}_{p}(s_{1},s_{2};d,g;\chi) \\
& \quad =\prod _{j=1}^{4}\left( 1-\frac{\alpha _{j}(p)\chi (p)}{p^{s_{1}+1/2}} \right)\prod _{j=1}^{4}\left( 1-\frac{\overline{\alpha _{j}(p)\chi (p)}}{p^{s_{2}+1/2}} \right) \\
& \quad \quad \times \left( |a_{\pi}(g_{p})|^{2}+\sum _{r=1}^{\infty}\frac{a_{\pi}(p^{\nu +r})\overline{a_{\pi}(p^{\nu})}\chi (p^{r})}{p^{(1/2+s_{1})r}} + \sum _{r=1}^{\infty}\frac{a_{\pi}(p^{\nu})\overline{a_{\pi}(p^{\nu +r})}\overline{\chi (p^{r})}}{p^{(1/2+s_{2})r}} \right).
\end{align*}
Due to the computation in Case II, on the set $\Re (s_{1})=\Re (s_{2})=100/\log Q$, we have 
\begin{equation}
\label{516}
\prod _{p|g}\left(\prod _{j=1}^{4}\left( 1-\frac{\alpha _{j}(p)\chi (p)}{p^{s_{1}+1/2}} \right)\prod _{j=1}^{4}\left( 1-\frac{\overline{\alpha _{j}(p)\chi (p)}}{p^{s_{2}+1/2}} \right) \right) \ll \tau (g)^{c_{1}}.
\end{equation}
Furthermore, 
\begin{equation}
\label{517}
\begin{aligned}
& \prod _{p|g} \left( |a_{\pi}(g_{p})|^{2}+\sum _{r=1}^{\infty}\frac{a_{\pi}(p^{\nu +r})\overline{a_{\pi}(p^{\nu})}\chi (p^{r})}{p^{(1/2+s_{1})r}} + \sum _{r=1}^{\infty}\frac{a_{\pi}(p^{\nu})\overline{a_{\pi}(p^{\nu +r})}\overline{\chi (p^{r})}}{p^{(1/2+s_{2})r}} \right) \\
& \quad \ll \prod _{p|g}\left\{ |a_{\pi}(g_{p})|^{2}\left(1+\frac{2}{p^{1/2+100/\log Q}}\right)+2\sum _{r=2}^{\infty}\frac{(\nu +r+1)^{4}(\nu +1)^{4}}{p^{(1/2+100/\log Q)r}}  \right\} \\
& \quad \ll \prod _{p|g}\left\{ |a_{\pi}(g_{p})|^{2}\left(1+\frac{2}{p^{1/2+100/\log Q}}\right) +\frac{2(\nu +1)^{8}c^{'}}{p^{1+200/\log Q}}  \right\} \\
& \quad \ll \prod _{p|g}\left\{ (\nu _{p,g}+1)^{8}\left(1+\frac{2}{p^{1/2+100/\log Q}} +  \frac{2 c^{'}}{p^{1+200/\log Q}}   \right)  \right\} \\
& \quad \ll \prod _{p|g}\left\{  (\nu _{p,g}+1)^{8}(3+2c^{'})  \right\} \\
& \quad \ll \tau (g)^{c_{3}+8},
\end{aligned}
\end{equation}
where $c^{'}$ is some absolute positive constant and  $c_{3}=2\log (3+2c^{'})/\log 2$. By (\ref{516}) and (\ref{517}), 
\begin{equation}
\label{518}
\prod _{p|g}{\cal M}_{p}(s_{1},s_{2};d,g;\chi)\ll \tau (g)^{c_{1}+c_{3}+8}.
\end{equation}
By (\ref{515}) and (\ref{518}), 
\begin{equation}
\label{519}
\theta (s_{1},s_{2};d,g;\chi):=\prod _{p|dg} {\cal M}_{p}(s_{1},s_{2};d,g;\chi)\ll \tau (d)^{c_{1}}\tau (g)^{c_{1}+c_{3}+8}.
\end{equation}
By (\ref{d}) and (\ref{519}), we obtain the conclusion of the lemma.
\end{proof}
We move the lines of integration in (\ref{53}) to $\Re (s_{1})=\Re (s_{2})=100/\log Q$. We do not cross the poles of the integrand and by Lemma 5.2 we have
\begin{equation}
\label{67}
\begin{aligned}
& {\cal G}_{\pi}(\Psi ,Q) \\
& \quad \ll N_{\pi}^{\frac{100}{\log Q}}Q\log Q\sum _{a\leq 2Q}\sum _{b>0}\sum _{h>0}\sum _{\chi (\mathrm{mod}\; abh)}\underset{b|g, (a,g)=1}{\sum _{g}}\underset{(d,gN_{\pi})=1}{\sum _{d\leq D}}\frac{\tau (d)^{c_{1}}\tau (g)^{c_{2}}}{adg \varphi (abh)} \\
& \quad \quad \times \int _{(\frac{100}{\log Q})} \int _{(\frac{100}{\log Q})} \left(| L(s_{1}+1/2, \pi \otimes \chi)L(s_{2}+1/2, \tilde{\pi}\otimes \overline{\chi})| +1 \right) \\
 & \hspace{8cm}  \times \left| \tilde{\cal W}^{\pm}\left( s_{1},s_{2}; \frac{QN_{\pi}^{1/2}d}{gh(\log Q)^{2\alpha}} \right) \right| |ds_{1}ds_{2}|.
\end{aligned}
\end{equation}
By (\ref{52}), for any positive integer $k$, we have 
\begin{align*}
&  \tilde{\cal W}^{\pm}\left( s_{1},s_{2}; \frac{QN_{\pi}^{1/2}d}{gh(\log Q)^{2\alpha}} \right) \\
& \quad \quad \quad    \ll \frac{1}{\max \{|s_{1}|,|s_{2}| \}^{k}}\left( 1 +\frac{QN_{\pi}^{1/2}d}{gh(\log Q)^{2\alpha}} \right)^{k-1}\exp \left(-c\left(\frac{QN_{\pi}^{1/2}d}{gh(\log Q)^{2\alpha}}\right)^{-\frac{1}{4}}  \right).
\end{align*}
Therefore, the sum over $g$ and $d$ in (\ref{67}) is 
\begin{align*}
& \underset{(a,g)=1}{\underset{b|g}{\sum _{g}}}\underset{(d,gN_{\pi})=1}{\sum _{d\leq D}}\frac{\tau (d)^{c_{1}}\tau (g)^{c_{2}}}{dg}\left| \tilde{W}^{\pm}\left(s_{1},s_{2}; \frac{QN_{\pi}^{1/2}d}{gh(\log Q)^{2\alpha}} \right) \right| \\
& \quad \ll \frac{1}{\max \{|s_{1}|, |s_{2}| \}^{k}}\underset{b|g}{\sum _{g}}\frac{\tau (g)^{c_{2}}}{g} \sum _{d\leq D}\frac{\tau (d)^{c_{1}}}{d}\left(1 +\frac{QN_{\pi}^{1/2}d}{gh(\log Q)^{2\alpha}} \right)^{k-1}\exp \left( -c \left(\frac{QN_{\pi}^{1/2}d}{gh(\log Q)^{2\alpha}}\right)^{-\frac{1}{4}} \right) \\
& \quad \ll \frac{\left( 1+ \frac{QN_{\pi}^{1/2}D}{bh(\log Q)^{2\alpha}}\right)^{k-1}}{\max \{|s_{1}|,|s_{2}| \}^{k}}(\log D)^{c_{3}}\underset{b|g}{\sum _{g}}\frac{\tau (g)^{c_{2}}}{g} \exp \left( -c \left( \frac{ gh(\log Q)^{2\alpha} }{QN_{\pi}^{1/2}D } \right)^{\frac{1}{4}} \right) \\
& \quad \ll \frac{\left( 1+ \frac{QN_{\pi}^{1/2}D}{bh(\log Q)^{2\alpha}}\right)^{k-1}}{\max \{|s_{1}|,|s_{2}| \}^{k}} \frac{(\log Q)^{\beta}\tau (b)^{c_{2}}}{b}\exp \left( -c \left( \frac{ bh(\log Q)^{2\alpha} }{QN_{\pi}^{1/2}D } \right)^{\frac{1}{4}} \right),
\end{align*}
where $c_{3}$ and $\beta$ above are some positive number, independent of $\alpha$, which might be replaced with larger value later. We estimate the integrals and summations with $S_{1}\leq |s_{1}| \leq 2S_{1}$, $S_{2}\leq |s_{2}| \leq 2S_{2}$, $A\leq a <2A$, $B\leq b <2B$, $H\leq h <2H$ in (\ref{67}). Since $\varphi (abh)^{-1}\ll (abh)^{-1}\log ^{\varepsilon}(abh)$, the contribution of this part to (\ref{67}) is bounded by 
\begin{align*}
& N_{\pi}^{\frac{100}{\log Q}}Q(\log Q)^{\beta} \sum _{A\leq a <2A}\sum _{B\leq b <2B}\sum _{H\leq h <2H}\sum _{\chi (\mathrm{mod}\; abh)}\frac{\log ^{\varepsilon}(ABH)}{A^{2}BH} \frac{\left( 1+ \frac{QN_{\pi}^{1/2}D}{BH(\log Q)^{2\alpha}}\right)^{k-1}}{\max \{|S_{1}|,|S_{2}| \}^{k}}  \\
& \quad \quad \times \frac{\tau (b)^{c_{2}}}{B} \exp \left(-c\left( \frac{ BH(\log Q)^{2\alpha} }{QN_{\pi}^{1/2}D } \right)^{\frac{1}{4}}  \right) \\
&\quad \times \int _{S_{1}\leq |s_{1}| \leq 2S_{1}}\int _{S_{2}\leq |s_{2}| \leq 2S_{2}} \left(1+|L(s_{1}+1/2, \pi \otimes \chi)L(s_{2}+1/2, \tilde{\pi}\otimes  \overline{\chi})| \right) |ds_{1}ds_{2}|.
\end{align*}
 Put $l=abh$. Then the above is at  most
\begin{equation}
\label{68}
\begin{aligned}
&N_{\pi}^{\frac{100}{\log Q}}Q(\log Q)^{\beta}\frac{\log ^{\varepsilon}(ABH)}{A^{2}B^{2}H} \frac{\left( 1+ \frac{QN_{\pi}^{1/2}D}{BH(\log Q)^{2\alpha}}\right)^{k-1}}{\max \{|S_{1}|,|S_{2}| \}^{k}}\exp \left(-c\left( \frac{ BH(\log Q)^{2\alpha} }{QN_{\pi}^{1/2}D } \right)^{\frac{1}{4}}  \right)   \\
& \quad \sum _{ABH \leq l <8ABH}\tau (l)^{c_{2}}\tau _{3}(l) \\
& \quad \times \sum _{\chi (\mathrm{mod}\; l)}\int _{S_{1}\leq |s_{1}| \leq 2S_{1}}\int _{S_{2}\leq |s_{2}| \leq 2S_{2}} \left(1+|L(s_{1}+1/2, \pi \otimes \chi)L(s_{2}+1/2, \tilde{\pi}\otimes  \overline{\chi})| \right) |ds_{1}ds_{2}|.
\end{aligned}
\end{equation}
By our assumption (\ref{assumption}), 
\begin{align*}
& \sum _{\chi (\mathrm{mod}\; l)}\int _{S_{1}\leq |s_{1}| \leq 2S_{1}}\int _{S_{2}\leq |s_{2}| \leq 2S_{2}} \left(1+|L(s_{1}+1/2, \pi \otimes \chi)L(s_{2}+1/2, \tilde{\pi}\otimes  \overline{\chi})| \right) |ds_{1}ds_{2}| \\
& \quad \leq \sum _{\chi (\mathrm{mod}\; l)}\int _{S_{1}\leq |s_{1}| \leq 2S_{1}}\int _{S_{2}\leq |s_{2}| \leq 2S_{2}} \left(1+|L(s_{1}+1/2, \pi \otimes \chi)|^{2}+|L(s_{2}+1/2, \tilde{\pi}\otimes  \overline{\chi})|^{2} \right) |ds_{1}ds_{2}| \\
& \quad \ll lN_{\pi}^{\varepsilon}S^{2}\log ^{K}(l(S+1)),
\end{align*}
where 
\[
S:=\max \{S_{1}, S_{2} \}.
\]
Thus (\ref{68}) is bounded by
\begin{equation}
\label{69}
\begin{aligned}
&N_{\pi}^{\varepsilon}Q(\log Q)^{\beta}\frac{H\log ^{K+\varepsilon}(ABH(S+1))\left( 1+ \frac{QN_{\pi}^{1/2}D}{BH(\log Q)^{2\alpha}} \right)^{k-1}}{S^{k-2}}\exp \left(-c \left( \frac{BH(\log Q)^{2\alpha}}{QN_{\pi}^{1/2}D}\right)^{\frac{1}{4}} \right).
\end{aligned}
\end{equation}
By (\ref{69}), (\ref{67}) is bounded by
\begin{equation}
\label{70}
\begin{aligned}
&N_{\pi}^{\varepsilon}Q(\log Q)^{\beta}\\
& \quad \times  \sum _{A,B,H,S}^{\quad \quad d}\frac{H\log ^{K+\varepsilon}(ABH(S+1))\left( 1+ \frac{QN_{\pi}^{1/2}D}{BH(\log Q)^{2\alpha}} \right)^{k-1}}{S^{k-2}}\exp \left(-c \left( \frac{BH(\log Q)^{2\alpha}}{QN_{\pi}^{1/2}D}\right)^{\frac{1}{4}} \right) \\
&  \leq N_{\pi}^{\varepsilon}Q(\log Q)^{\beta}\\
& \quad \times \sum _{B,H,S}^{\quad \quad d}\frac{H\log ^{K+\varepsilon}(QBH(S+1))\left( 1+ \frac{QN_{\pi}^{1/2}D}{BH(\log Q)^{2\alpha}} \right)^{k-1}}{S^{k-2}}\exp \left(-c \left( \frac{BH(\log Q)^{2\alpha}}{QN_{\pi}^{1/2}D}\right)^{\frac{1}{4}} \right).
\end{aligned}
\end{equation}
We take $k=1$ if $S\leq 1+QN_{\pi}^{1/2}D/BH(\log Q)^{2\alpha}$, and otherwise take $k=4$. Hence the contribution of the part with $S\leq 1+QN_{\pi}^{1/2}D/BH(\log Q)^{2\alpha}$ is at most
\begin{align*}
& N_{\pi}^{\varepsilon}Q(\log Q)^{\beta}\sum _{B,H}^{\quad \quad d}H\log ^{\beta}(QN_{\pi}BH)\left( 1+ \frac{QN_{\pi}^{1/2}D}{BH(\log Q)^{2\alpha}} \right)\exp \left(-c\left( \frac{BH(\log Q)^{2\alpha}}{QN_{\pi}^{1/2}D} \right)^{\frac{1}{4}} \right) \\
& \quad \ll \frac{Q^{2}DN_{\pi}^{1/2+\varepsilon}(\log N_{\pi}Q)^{\beta}}{(\log Q)^{2\alpha}}.
\end{align*}
Also, the contribution of the part with $S>1+QN_{\pi}^{1/2}D/BH(\log Q)^{2\alpha}$ has the same upper bound.   Therefore, we arrive at the following conclusion.
\begin{prop}
We have 
\begin{equation}
\label{g}
{\cal G}_{\pi}(\Psi ,Q) \ll \frac{Q^{2}DN_{\pi}^{1/2+\varepsilon}(\log N_{\pi}Q)^{\beta}}{(\log Q)^{2\alpha}},
\end{equation}
where $\beta >0$ is some constant which is independent of $\alpha$. 
\end{prop}


\section{Completion of the proof}
We let 
\[
N_{\pi}\ll (\log Q)^{N},
\]
where $N$ is an arbitrarily fixed positive number. Then by (\ref{g}) we have 
\[
{\cal G}_{\pi}(\Psi ,Q)\ll \frac{Q^{2}D}{(\log Q)^{2\alpha -\beta -N/2-\varepsilon N}}.
\]
Recall that $D$ is given by $D=(\log Q)^{\delta}$. We take $\delta$ and $\alpha$ so that the inequalities $\delta >A+1$ and $2\alpha -\beta -N/2-\delta -\varepsilon N \geq 1$ hold simultaneously, where $A$ is the constant in (\ref{S}).  Then we have
\[
{\cal S}_{\pi}(\Psi ,Q)+{\cal G}_{\pi}(\Psi ,Q) \ll \frac{Q^{2}}{\log Q}.
\]
Summing up, we obtain the following result.
\begin{prop}
Suppose the conductor $N_{\pi}$ of the automorphic representation $\pi$ is bounded by $N_{\pi}\ll (\log Q)^{N}$ for some positive constant $N$. Then under the assumption of Theorem 1.1, we have 
\begin{equation}
\label{SG}
{\cal S}_{\pi}(\Psi ,Q)+{\cal G}_{\pi}(\Psi ,Q) \ll \frac{Q^{2}}{\log Q}.
\end{equation}
The implied constant might be dependent on $N$, but is independent of $N_{\pi}$ and $Q$. 
\end{prop}
Combining  these results we obtain (\ref{maintheorem}).  \hspace{6.8cm}  \fbox


\section{Acknowledgements}
This work is partially supported by the JSPS, KAKENHI Grant Number 21K03204.


\noindent
Kobe University, \\
Rokkodai, Nada,\\
Hyogo, Japan\\
E-mail address: souno@math.kobe-u.ac.jp

\end{document}